\documentclass[11pt, a4paper]{article}

\usepackage{graphicx,tikz,caption,subcaption,fontenc}
\usepackage{fullpage}
\usepackage{hyperref}
\usepackage{enumitem,verbatim}
\usepackage{amsmath}
\usepackage{amssymb}
\usepackage{amsfonts}
\usepackage{amsthm}
\usepackage[capitalise]{cleveref}
\usepackage{bm}

\usepackage[margin=2.5cm]{geometry}

\newtheorem{theorem}{Theorem}[section]
\newtheorem{lemma}[theorem]{Lemma}
\newtheorem{claim}[theorem]{Claim}

\newtheorem{conjecture}[theorem]{Conjecture}
\newtheorem{observation}[theorem]{Observation}
\newtheorem*{observation*}{Observation}
\newtheorem{proposition}[theorem]{Proposition}
\newtheorem{problem}[theorem]{Problem}
\newtheorem{question}[theorem]{Question}
\newtheorem*{question*}{Question}
\newtheorem{innerdefinition}{Definition}
\newenvironment{definition*}
  {
   \innerdefinition}
  {\endinnerdefinition}

\theoremstyle{definition}
\newtheorem{defn}[theorem]{Definition}

\theoremstyle{remark}

\newcounter{propcounter}

\newenvironment{poc}{\begin{proof}[Proof of claim]}{\end{proof}}

\newcommand{\C}[1]{{\protect\mathcal{#1}}}

\newcommand{\ex}{\mathrm{ex}}

\usepackage{tikz}
\tikzstyle{aNode} = [circle, fill = black]
\tikzstyle{bNode} = [circle,draw = black, thick]

\newcommand{\ppoints}[1]{%
\begin{tikzpicture}[inner sep = 0.7pt, #1]%
\node (1) at (0,-2) [aNode]{};
\node (3) at (1.5,-2) [aNode]{};
\node (2) at (0.75,-1) [aNode]{};
\end{tikzpicture}%
}

\def\points{\ppoints{scale=0.08}}

\usetikzlibrary{patterns}
\usetikzlibrary{shapes}
\usetikzlibrary{decorations.pathreplacing}
\usetikzlibrary{decorations.pathmorphing}
\usetikzlibrary{positioning}

\usepackage{todonotes}

\usepackage{xcolor}

\title{Vanishing orders, suspensions and zero degree Tur\'an densities}
\author{Jiangdong Ai\thanks{School of Mathematical Sciences and LPMC, Nankai University, Tianjin 300071, China. Supported by the National Natural Science Foundation of China (No.12522117). Email:\texttt {jd@nankai.edu.cn}.}
\and
Laihao Ding\thanks{School of Mathematics and Statistics, and Key Laboratory of Nonlinear Analysis \& Applications (Ministry of Education), Central China Normal University, Wuhan 430079, China. Supported by the Fundamental Research Funds for the Central Universities (CCNU25JCPT031, XJ2026006601). Email: \texttt{dinglaihao@ccnu.edu.cn}.}
\and 
Hong Liu\thanks{Extremal Combinatorics and Probability Group (ECOPRO),  Institute for Basic Science (IBS), Daejeon, South Korea. Supported by IBS-R029-C4. Email: \texttt{hongliu@ibs.re.kr}.}
\and Haotian Yang\thanks{Department of Mathematics, California Institute of Technology, Pasadena, USA. Supported by Seed Fund Program for International Research Cooperation of Shandong University. Email: \texttt{hyang3@caltech.edu}.}}

\begin{document}
\date{}
\maketitle

\begin{abstract}
For integers $1\le \ell<k$, the $\ell$-degree Tur\'an density $\pi_\ell(F)$ measures the minimum $\ell$-degree threshold that forces a copy of a fixed $k$-uniform hypergraph $F$, generalizing both the classical Tur\'an density $\pi_1$ and the codegree Tur\'an density $\pi_{k-1}$. Motivated by Erd\H{o}s' characterization of $k$-graphs with zero Tur\'an density, we study the structural implications of vanishing $\ell$-degree Tur\'an density.

Our main result concerns the case $\ell=2$. We prove that, for every $k\ge3$, if a $k$-graph $F$ satisfies $\pi_2(F)=0$, then $F$ admits a $2$-vanishing order, that is, a global vertex ordering under which all edges align canonically with respect to their pairs. This extends to all uniformities a structural phenomenon previously known for $3$-graphs, and gives a higher-degree analogue of the classical fact that $\pi_1(F)=0$ forces $F$ to be $k$-partite. In particular, the absence of a $2$-vanishing order is a structural obstruction to vanishing $2$-degree Tur\'an density.

We also establish a suspension principle connecting consecutive degree parameters. Given a $(k-1)$-graph $F$, let $\mathcal S_F$ be the $k$-graph obtained by adding an apex vertex $v$ and replacing each edge $e\in E(F)$ with $v\cup e$. We show that, for $2\le \ell<k$,
$\pi_{\ell}(\mathcal S_F)=0$ if and only if $\pi_{\ell-1}(F)=0$.
This provides a bridge between different degree Tur\'an densities and allows vanishing results to be lifted across uniformities and degree parameters.
As an application, we prove that except the classical Tur\'an density, all other degree Tur\'an densities accumulate at zero. 

The proof of our main result combines random geometric building blocks, a design-theoretic gluing scheme, and random sparsification to reconcile positive $2$-degree with local vanishing structure. 
\end{abstract}


\section{Introduction}
A central topic in extremal combinatorics is the \emph{Tur\'an number} 
$\ex(n,F)$, the maximum number of edges in an $n$-vertex $k$-uniform hypergraph (or \emph{$k$-graph}) avoiding a fixed $k$-graph $F$.  
Passing to the limit, the \emph{Tur\'an density} is defined by
$$\pi(F):=\lim_{n\to\infty}\frac{\ex(n,F)}{\binom nk}.$$
For graphs, that is, when $k=2$, the Tur\'an density $\pi(F)$ is determined for every graph $F$ by the celebrated Erd\H{o}s–Stone–Simonovits Theorem.
For $k\ge 3$, however, the situation is radically different. Even Tur\'an's original problem of determining the density of the complete $3$-graph $K_4^{(3)}$, posed in 1941, remains open. 
This difficulty has led to extensive study of hypergraph Tur\'an-type problems; see, for example,~\cite{l2norm,Erdos-sos,reihersurvey}.

\subsection{$\ell$-degree Tur\'an density}

A natural refinement of Tur\'an density replaces global edge density with local degree conditions. We first recall the relevant terminology.  
Let $H$ be a $k$-graph. For $S\subseteq V(H)$ with $|S|<k$, the \emph{link} of $S$, denoted by $L_H(S)$, is the $(k-|S|)$-graph on $V(H)\setminus S$ with edge set $\{e\setminus S: S\subset e\in E(H)\}$. 
The \emph{degree} of $S$ is $d_H(S)=|L_H(S)|$, and the \emph{minimum $\ell$-degree} $\delta_\ell(H)$ is the minimum of $d_H(S)$ over all $\ell$-subsets $S$.

Given a $k$-graph $F$, the \textit{$\ell$-degree Tur\'{a}n number} $\ex_{\ell}(n,F)$ is the maximum $\delta_{\ell}(H)$ that an $n$-vertex $F$-free $k$-graph $H$ can admit. The corresponding \textit{$\ell$-degree Tur\'{a}n density} is 
$$\pi_{\ell}(F)=\lim\limits_{n\to \infty}\frac{\ex_{\ell}(n,F)}{\binom{n-\ell}{k-\ell}}.$$ 
This limit always exists \cite{l-degree}, and for every $k$-graph $F$,
$$\pi_{k-1}(F)\leq\pi_{k-2}(F)\leq\cdots\leq\pi_{2}(F)\leq\pi_{1}(F)=\pi(F).$$
The notion of $\ell$-degree Tur\'an density generalizes both the classical Tur\'an density $\pi_1$ and the codegree Tur\'an density $\pi_{k-1}$.

\medskip

Determining $\pi_\ell(F)$ for $\ell\ge2$ is highly challenging, and most exact results concern $3$-graphs. 
Nagle~\cite{K43-nagle} and Czygrinow and Nagle~\cite{K43nagle} conjectured that 
$\pi_{2}(K_4^{(3)-}) = \tfrac14$ and $\pi_{2}(K_4^{(3)}) = \tfrac12$, respectively; 
the former was recently confirmed using flag algebras~\cite{K43-codegree}.  
Other $3$-graphs with determined codegree Tur\'an densities include the Fano plane~\cite{mubayi2005co}, 
$F_{3,2}$~\cite{falgas2015codegree}, 
tight cycles $C_r^{(3)}$ for $r\geq 10$~\cite{ma2024codegree,piga2024codegree}, 
and $C_r^{(3)-}$ for $r\ge5$~\cite{piga2022codegree}; 
see also~\cite{Luo-cycle,keevash2007codegree,lo2018codegree,Rong-cycle,James,zhang2021codegree}.

For $k\ge3$, the only general zero-density characterization remains the classical theorem of Erd\H{o}s~\cite{erdos1964extremal}, which states that $\pi(F)=0$ if and only if $F$ is $k$-partite. 
This naturally motivates the corresponding problem for higher degree Tur\'an densities.

\begin{problem}\label{chrac}
For integers $k>\ell\ge2$, characterize the $k$-graphs $F$ with $\pi_\ell(F)=0$.
\end{problem}

Even for $3$-graphs, Problem~\ref{chrac} appears highly nontrivial. 
Recent work of Lamaison, Wang and the last three authors~\cite{ding20243} revealed a close connection with the \emph{uniform} Tur\'an density $\pi_{\points}(F)$, defined as the supremum over all $d$ for which there exist infinitely many $F$-free $k$-graphs whose every induced linear-size subhypergraph has edge density at least $d$. 
In particular, for $3$-graphs it was shown in~\cite{ding20243} that $\pi_{2}(F)=0$ implies $\pi_{\points}(F)=0$. 
A deep theorem of Reiher, R\"{o}dl and Schacht~\cite{reiher2018hypergraphs} further gives a structural characterization of the latter condition: a $3$-graph has $\pi_{\points}(F)=0$ if and only if it admits a \textit{$2$-vanishing order}.

This suggests that vanishing degree Tur\'an density may in general force a rigid global ordering structure on the forbidden configuration. 
We formalize this through the following notion. Roughly speaking, an \emph{$\ell$-vanishing order} is a vertex ordering under which every edge behaves canonically with respect to its $\ell$-subsets, so that these subsets always occupy identical relative positions. 
Here an \emph{ordering} of a finite set $S$ is a bijection $\sigma:S\to [|S|]$. 
Denote by $S^{(t)}$ the collection of all $t$-subsets of a set $S$. 
The formal definition follows.

\begin{defn}[$\ell$-vanishing order]\label{order}
 Given integers $1\leq \ell\leq k$ and a \(k\)-graph \(F\), we say that an ordering $\sigma$ of \(V(F) \) is an \textit{\(\ell\)-vanishing order} of $F$ if there is a coloring
$$\varphi : V(F)^{(\ell)} \to \left\{\mathcal{L} : \mathcal{L} \in [k]^{(\ell)}\right\}$$
such that for every edge \(e = \{v_{1}, v_{2}, \ldots, v_{k}\}\in E(F)\) with \(\sigma(v_1) < \sigma(v_2) < \cdots < \sigma(v_k)\), it holds that
$$\varphi\left( \{v_{i} : i \in \mathcal{L}\}\right) = \mathcal{L}$$
 for every $\mathcal{L} \in [k]^{(\ell)}$. 
In particular, we call $\varphi$ an \textit{$\ell$-vanishing coloring} of  $F$ under $\sigma$.
\end{defn}

A basic observation is that a $k$-graph has a $1$-vanishing order if and only if it is $k$-partite, while every $k$-graph trivially admits a $k$-vanishing order. 
Combining the above results of~\cite{ding20243,reiher2018hypergraphs} gives the following structural consequence for $3$-graphs.

\begin{theorem}[\cite{ding20243,reiher2018hypergraphs}]\label{zeroco}
Let $F$ be a $3$-graph. If $\pi_{2}(F)=0$, then it has a $2$-vanishing order.  
\end{theorem}

\subsection{Our results}
Our main result extends \cref{zeroco} from $3$-graphs to all uniformities. It shows that vanishing $2$-degree Tur\'an density forces the same ordering structure in every $k$-uniform setting.

\begin{theorem}\label{2-vanishing}
Let $k\geq3$ and $F$ a $k$-graph. If $\pi_2(F)=0$, then it has a $2$-vanishing order. 
\end{theorem}

\cref{2-vanishing} may be viewed as a higher-degree analogue of the classical fact that if the Tur\'an density $\pi_1(F)=0$, then $F$ must have a 1-vanishing order (i.e.~be $k$-partite). It shows that the absence of a $2$-vanishing order is a structural obstruction to vanishing $2$-degree Tur\'an density. More conceptually, the theorem reflects an interesting local--global forcing phenomenon. 
The condition $\pi_2(F)=0$ is a local embedding condition imposed on host hypergraphs. It forces a rigid global ordering structure inside the forbidden configuration.

\cref{2-vanishing} suggests that vanishing orders capture the structural core of degree Tur\'an zero-density phenomena, motivating the following conjecture.

\begin{conjecture}\label{general}
Let $k\geq 4, \ell\geq3$ be integers satisfying that $\ell<k$. Then any $k$-graph $F$ with $\pi_{\ell}(F)=0$ has an $\ell$-vanishing order.     
\end{conjecture}

In \cref{example}, we show that \cref{general} is best possible in the sense that the $\ell$-vanishing order cannot in general be replaced by an $(\ell-1)$-vanishing order. 
As additional evidence, we give the following weaker necessary condition for all $\ell$.

\begin{theorem}\label{j-vanishing}
Let $F$ be a $k$-graph with $\pi_{\ell}(F)=0$. Then 
\begin{enumerate}
[label=\rm(\alph*), ref=(\alph*)]          
\item $F$ has an $(\ell+1)$-vanishing order, and \label{cond1}
\item $L_F(S)$ is a $(k-\ell+1)$-partite $(k-\ell+1)$-graph for every $S\in V(F)^{(\ell-1)}$. \label{cond2}
\end{enumerate}
\end{theorem}

Observe that any $\ell$-vanishing order is automatically an $(\ell+1)$-vanishing order. 
Moreover, as shown in \cref{example}, $k$-graphs admitting $\ell$-vanishing orders satisfy \ref{cond2}. 
When $\ell\in\{1,k-1\}$, condition~\ref{cond1} is redundant since it follows from~\ref{cond2}. 
For all other values of $\ell$, we provide examples in \cref{example} showing that neither \ref{cond1} nor \ref{cond2} implies the other.

We next state a suspension result, which plays a supporting but important role. Given a $(k-1)$-graph $F$, the \emph{suspension} of $F$, denoted by $\mathcal S_F$, is the $k$-graph obtained from $F$ by adding an apex vertex $v$ and replacing every edge $e\in E(F)$ by $e\cup\{v\}$. The following theorem transfers vanishing from $(\ell-1)$-degree Tur\'an density in uniformity $k-1$ to $\ell$-degree Tur\'an density in uniformity $k$. 

\begin{theorem}\label{suspension}
Let $2\le \ell<k$ and $F$ a $(k-1)$-graph.  
Then $ \pi_\ell(\mathcal{S}_F)=0$ if and only if $\pi_{\ell-1}(F)=0$.
\end{theorem}

\cref{suspension} implies that if if a $k$-graph $F$ satisfies $\pi_\ell(F)=0$, then $\pi_{\ell-1}(L_F(v))=0$ for every $v\in V(F)$.
Thus \cref{suspension} provides a bridge between consecutive degree parameters and is used below to propagate the vanishing phenomena obtained from \cref{2-vanishing}.

We finally present an application of the above results. 
Let $\Pi_{\ell}^{k}=\{\pi_{\ell}(F) : F \text{ is a $k$-graph}\}$ denote the set of all possible $\ell$-degree Tur\'an densities of $k$-graphs. 
Given the difficulty of determining individual values of $\pi_\ell(F)$, a natural alternative is to study the global structure of $\Pi_{\ell}^{k}$.
For $\ell=1$, this direction is closely related to the classical Erd\H{o}s jumping conjecture and remains highly challenging. 
It was only recently shown that $\Pi_{1}^{k}$ admits an accumulation point \cite{Accu-Turan-Bjarne}. Moreover, by the theorem of Erd\H{o}s~\cite{erdos1964extremal}, $0$ is not an accumulation point of $\Pi_{1}^{k}$. 
In contrast, Piga and Sch\"{u}lke~\cite{piga2023hypergraphs} recently proved that $0$ is an accumulation point of $\Pi_{k-1}^{k}$, providing the first example among Tur\'an-type densities that does not jump at zero. 
This naturally leads to the following question.

\begin{question}\label{zero}
For integers $2\leq \ell\leq k-2$, is zero an accumulation point of $\Pi_{\ell}^{k}$? 
\end{question}

Combining \cref{2-vanishing} with \cref{suspension}, we answer this question affirmatively. Thus, except for the classical Tur\'an density, all other degree Tur\'an densities accumulate at zero.

\begin{theorem}\label{j-accumulation}
For all integers $k>\ell\geq 2$, zero is an accumulation point of $\Pi_{\ell}^k$.    
\end{theorem}

\subsection{Proof idea}\label{sec:pf-idea}
We first discuss the main idea behind \cref{2-vanishing}.
The contrapositive of \cref{2-vanishing} says that if a $k$-graph $F$ admits no $2$-vanishing order, then $\pi_2(F)>0$. A natural strategy would be to construct $k$-graphs that are $2$-vanishing (and hence $F$-free) while having positive minimum $2$-degree. However, a direct construction is impossible: a \emph{global} $2$-vanishing order is incompatible with positive minimum $2$-degree, as shown below.

\begin{observation}\label{diff}
Let $2\le \ell<k$ and let $H$ be a $k$-graph on $n$ vertices that admits an $\ell$-vanishing order. Then $\delta_{\ell}(H)\leq\binom{n-\ell-1}{k-\ell-1}=o(n^{k-\ell})$.
\end{observation}
\begin{proof}
Without loss of generality, let $V(H)=[n]$ and suppose the natural order is an $\ell$-vanishing order of $H$. If $L_H([\ell])=\varnothing$, then $\delta_\ell(H)=0$ and we are done. Otherwise, there exists a $(k-\ell)$-set $S$ such that $e:=[\ell]\cup S\in E(H)$. Fix $s\in S$ and set $X=\{2,3,\dots,\ell,s\}$.  

In the edge $e$, the smallest vertex is $1\notin X$. Hence, by the $\ell$-vanishing property, every edge containing $X$ must also have its smallest vertex outside $X$. In particular, every such edge must contain $1$. Therefore the number of edges containing $X$ is at most $\binom{n-\ell-1}{k-\ell-1},$ as desired.
\end{proof}

To overcome this obstruction, we replace the global requirement by a \emph{local} one. We construct $k$-graphs with positive minimum $2$-degree such that every induced subhypergraph on a bounded number of vertices is $2$-vanishing (see \cref{vanishing2}). This balances two competing features: vanishing structure, which forces sparsity, and positive minimum $2$-degree, which enforces density.

The construction combines three ingredients.

\smallskip
\noindent\emph{(1) Random geometric building blocks.}
Extending ideas from an earlier work~\cite{ding20243} of the last three authors on $3$-graphs, we construct random $(2,1^{k-2})$-type $k$-graphs whose distinguished part carries a geometric random graph structure. This guarantees that every bounded induced subhypergraph admits a suitable cluster ordering and is therefore $2$-vanishing. Within a single block, all but $k-2$ missing pair types already have positive codegree .

\smallskip
\noindent\emph{(2) Design-theoretic gluing scheme.}
In the $3$-uniform case, the missing pair type can be handled by cyclically gluing three pieces of building blocks. Such a simple gluing no longer works for higher uniformity. We instead use a $(k-1)$-uniform combinatorial design to glue many $(2,1^{k-2})$-type blocks so that every pair of vertex parts is covered in a controlled manner, ensuring that all pair types acquire positive codegree. The additional structural complexity arising in higher uniformity necessitates this design-theoretic gluing.

\smallskip
\noindent\emph{(3) Random sparsification.}
After gluing, links of vertices may mix edges from different blocks, potentially destroying local $2$-vanishing. We therefore perform a random sparsification that separates these link structures while preserving positive minimum $2$-degree with high probability. This additional step is another new ingredient required in higher uniformity.

To prove \cref{suspension}, we start with an $F$-free $(k-1)$-graph $H$ of positive minimum $(\ell-1)$-degree, and for every $x\in[n]$ independently assign a random copy $H_x$ of $H$ on vertex set $[n]\setminus\{x\}$. 
Then a $k$-set $e\subseteq[n]$ is declared to be an edge if and only if $e\setminus\{x\}\in E(H_x)$ for every $x\in e$. 
This makes each vertex link a subhypergraph of an $F$-free $(k-1)$-graph, so the resulting $k$-graph is $\mathcal{S}_F$-free. The independent random relabellings make the local constraints sufficiently transverse, and bounded-differences estimates show that every $\ell$-set still has many completions.

To prove \cref{j-accumulation}, we first construct a sequence of $k$-graphs $(F_i)_{i\geq1}$ with no $2$-vanishing order, whose $2$-degree Tur\'an densities tend to zero (see \cref{lowdensity}). By \cref{2-vanishing}, these densities are positive, and hence zero is an accumulation point of $\Pi_2^k$. The suspension theorem then serves as a bridge from $\pi_{\ell-1}$ to $\pi_\ell$, allowing us to deduce \cref{j-accumulation} for all $\ell\ge2$ by induction.
The construction of the sequence of $k$-graphs $(F_i)_{i\geq1}$ is inspired by \cref{diff}.
For a fixed integer $m$, we take the \textit{tensor product} of all $m$-vertex $k$-graphs whose minimum $2$-degree is larger than $\binom{m-3}{k-3}$. The tensor product operation drives the corresponding $\ell$-degree Tur\'an density down, while the minimum degree condition of each coordinate $k$-graph ensures that the resulting hypergraph admits no $2$-vanishing order.

\medskip

\noindent{\bf Notations.~}Let $F$ be a $k$-graph. 
We use $F(t)$ to denote the \textit{$t$-blowup} of $F$, that is, a $k$-graph obtained from $F$ by replacing each vertex with a vertex class of size $t$ and replacing each edge with a complete $k$-partite $k$-graph on the corresponding vertex classes.
The following operation provides us a natural way to merge several orderings into a larger one. Let $S_1,S_2$ be two disjoint finite sets and let $\sigma_1,\sigma_2$ be two orderings of $S_1,S_2$ respectively. The sum of $\sigma_1$ and $\sigma_2$, denoted by $\sigma_1 \oplus \sigma_2$, is an ordering of $S_1\cup S_2$ where
\[\sigma_1 \oplus \sigma_2(s)= \left \{
\begin{array}{ll}
 \sigma_1(s),  & s\in S_1;\\
 \sigma_2(s)+|S_1|, & s\in S_2.
\end{array}
\right.\]
For more than two orderings $\sigma_1,\sigma_2,\ldots,\sigma_k$, the sum of them, denoted by $\sum_{i=1}^{k}\sigma_i$, is inductively defined by 
\[\sum_{i=1}^{k}\sigma_i=\left(\sum_{i=1}^{k-1}\sigma_i\right)\oplus \sigma_k.\]
Reversely, given an ordering $\sigma$ of a finite set $S$ and a subset $T\subseteq S$, we naturally obtain an ordering of $T$,  denoted by $\sigma|_T$, by letting $\sigma|_T(x)<\sigma|_T(y)$ if $\sigma(x)<\sigma(y)$ for all $x,y\in T$.

\section{Hypergraphs with vanishing orders}\label{example}
In this section, we develop the following recursive characterization of vanishing orders via links, and give two families of hypergraphs with(out) specific vanishing orders. This link-based viewpoint is more convenient for local arguments.

\begin{lemma}\label{vanishingorder}
Let $F$ be a $k$-graph, and $\sigma$ an ordering of $V(F)$. Then for any $1\leq j<\ell\leq k$, $\sigma$ is an $\ell$-vanishing order of $F$ if and only if for every $S\in V(F)^{(\ell-j)}$, $\sigma|_{V(L_F(S))}$ is a $j$-vanishing order of $L_F(S)$.   
\end{lemma}

In this paper, the following particular instance of \cref{vanishingorder} will be especially useful.

\begin{lemma}\label{1vanishing}
Let $F$ be a $k$-graph and $\sigma$ an ordering of $V(F)$. Then for any $2\leq\ell\leq k$, $\sigma$ is an $\ell$-vanishing order of $F$ if and only if for every $S\in V(F)^{(\ell-1)}$, $\sigma|_{V(L_F(S))}$ is a $1$-vanishing order of $L_F(S)$.  
\end{lemma}

Recall that a $k$-graph has a $1$-vanishing order if and only if it is $k$-partite. 
Hence, by \cref{1vanishing}, any $k$-graph admitting an $\ell$-vanishing order satisfies condition~\ref{cond2} of \cref{j-vanishing}.


\subsection{Proof of the equivalent formulation}
Let $\sigma$ be an ordering of a finite set $S$, and 
$X=\{x_1,x_2,\ldots,x_s\}\subseteq S$ 
with $\sigma(x_1)<\sigma(x_2)<\ldots<\sigma(x_s)$.
For any subset $Y\subseteq S$, let 
$\boldsymbol{I}_{X,\sigma}(Y)\in\mathbb{Z}^{s}$ denote the vector whose $i$-th coordinate counts the number of elements $y\in Y$ with $\sigma(x_{i-1})<\sigma(y)<\sigma(x_i)$. 
The proof of~\cref{vanishingorder} builds on the following observation.

\begin{proposition}\label{index}
Let $F$ be a $k$-graph, and $\sigma$ an ordering of $V(F)$. Then for any $1\leq\ell\leq k$, $\sigma$ is an $\ell$-vanishing order of $F$ if and only if 
$\boldsymbol{I}_{S,\sigma}(e)= \boldsymbol{I}_{S,\sigma}(e')$ for every $S\in V(F)^{(\ell)}$ and every pair of edges $e,e'\in E(F)$ containing $S$.
\end{proposition}

\begin{proof}
Fix $1\le \ell\le k$. Let $S=\{v_1,\dots,v_\ell\}\subseteq V(F)$ with
$\sigma(v_1)<\cdots<\sigma(v_\ell)$, and let $e\in E(F)$ contain $S$.
Write $e=\{w_1,\dots,w_k\}$ with $\sigma(w_1)<\cdots<\sigma(w_k)$, and define
\[
P_{S,\sigma}(e):=\{\, i\in[k]: w_i\in S\,\}\in [k]^{(\ell)},
\]
the set of positions occupied by $S$ inside $e$. By definition, $\boldsymbol I_{S,\sigma}(e)\in\mathbb Z^\ell$ records the numbers
\[
a_1:=|\{y\in e: \sigma(y)<\sigma(v_1)\}|,\qquad 
a_i:=|\{y\in e: \sigma(v_{i-1})<\sigma(y)<\sigma(v_i)\}|\ \ (2\le i\le \ell),
\]
i.e.\ it counts vertices of $e\setminus S$ lying before $v_1$ and between consecutive $v_{i-1},v_i$.

We claim that for fixed $(k,\ell)$, the data $P_{S,\sigma}(e)$ and $\boldsymbol I_{S,\sigma}(e)$
determine each other.
Indeed, if $P_{S,\sigma}(e)=\{i_1<\cdots<i_\ell\}$, then $\boldsymbol I_{S,\sigma}(e)=(i_1-1, i_2-i_1-1, \dots, i_\ell-i_{\ell-1}-1).$
Conversely, given $\boldsymbol I_{S,\sigma}(e)=(a_1,\dots,a_\ell)$ we recover $i_1=a_1+1$, $i_2=a_1+a_2+2$, \dots, $i_\ell=a_1+\cdots+a_\ell+\ell$.

Therefore, for fixed $S$ and all $e,e'\supseteq S$, the condition
$\boldsymbol I_{S,\sigma}(e)=\boldsymbol I_{S,\sigma}(e')$
is equivalent to
$P_{S,\sigma}(e)=P_{S,\sigma}(e')$. Now $\sigma$ is an $\ell$-vanishing order precisely when there exists a map
$\varphi:V(F)^{(\ell)}\to [k]^{(\ell)}$ such that for every edge
$e=\{w_1,\dots,w_k\}$ in $\sigma$-increasing order and every $\C L\in[k]^{(\ell)}$ we have
$\varphi(\{w_i:i\in \C L\})=\C L$. Equivalently, for every $\ell$-set $S$ contained in some edge we must have
$\varphi(S)=P_{S,\sigma}(e)$ for all $e\supseteq S$, which is possible if and only if the above
position set is well-defined. This yields the desired equivalence.
\end{proof}

\begin{proof}[Proof of~\cref{vanishingorder}]
Suppose first that $\sigma$ is an $\ell$-vanishing order of $F$. 
If there exists a subset $S\in V(F)^{(\ell-j)}$ such that $\sigma|_{V(L_F(S))}$ is not a $j$-vanishing order of $L_F(S)$, then by~\cref{index} there is a $j$-subset $R\subseteq V(F)\setminus S$ and two edges $e,e'\in L_F(S)$ containing $R$ such that 
$\boldsymbol{I}_{R,\sigma}(e)\neq\boldsymbol{I}_{R,\sigma}(e')$. 
Then for the $\ell$-set $S\cup R$ and the two edges $e\cup S$ and $e'\cup S$ of $F$, it follows that 
$\boldsymbol{I}_{S\cup R,\sigma}(e\cup S)\neq\boldsymbol{I}_{S\cup R,\sigma}(e'\cup S)$, 
contradicting the assumption that $\sigma$ is an $\ell$-vanishing order of $F$.  

Conversely, suppose that $\sigma$ satisfies that $\sigma|_{V(L_F(S))}$ is a $j$-vanishing order of $L_F(S)$ for every $S\in V(F)^{(\ell-j)}$. 
If $\sigma$ is not an $\ell$-vanishing order of $F$, then there exist an $\ell$-set $R\subseteq V(F)$ and two edges $e,e'$ containing $R$ such that 
$\boldsymbol{I}_{R,\sigma}(e)\neq\boldsymbol{I}_{R,\sigma}(e')$. 
Let $t$ be the smallest index for which the $t$-th coordinates of 
$\boldsymbol{I}_{R,\sigma}(e)$ and $\boldsymbol{I}_{R,\sigma}(e')$
differ. 
Write $R=\{v_1,v_2,\ldots,v_{\ell}\}$ with 
$\sigma(v_1)<\sigma(v_2)<\ldots<\sigma(v_{\ell})$. 
Define
\[
S=
\begin{cases}
\{v_1,\ldots,v_{\ell-j}\}, & \text{if } t>\ell-j,\\
\{v_1,\ldots,v_{t-1},v_{t+1},\ldots,v_{\ell-j+1}\}, & \text{if } t\le \ell-j.
\end{cases}
\]
Then $S$ is an $(\ell-j)$-subset of $R$, and $R\setminus S$ is a $j$-subset. 
A direct verification shows that 
\[
\boldsymbol{I}_{R\setminus S,\sigma}(e\setminus S)
\neq
\boldsymbol{I}_{R\setminus S,\sigma}(e'\setminus S),
\]
so $\sigma|_{V(L_F(S))}$ is not a $j$-vanishing order of $L_F(S)$, a contradiction.  
\end{proof}


\subsection{Two examples}

Given an $\ell$-graph $F$, the \textit{$k$-expansion} $F^+$ of $F$ is the $k$-graph obtained from $F$ by enlarging each edge of $F$ with a vertex subset of size $k-\ell$ disjoint from $V(F)$ such that distinct edges are enlarged by disjoint subsets.  

\begin{observation}\label{example1}
For any integer $r>k$, the $k$-expansion $K_{r}^{(\ell)+}$ of $K_{r}^{(\ell)}$ satisfies the following.
\begin{itemize}
\item $\pi_{\ell}(K_{r}^{(\ell)+})=0$;
\item $K_{r}^{(\ell)+}$ has an $\ell$-vanishing order;
\item when $\ell\geq2$, there exists an $(\ell-2)$-subset $S$ of $V(K_{r}^{(\ell)+})$ whose link is not $(k-\ell+2)$-partite.
\end{itemize}
\end{observation}

\begin{proof}
We first verify that $\pi_{\ell}(K_{r}^{(\ell)+})=0$. Indeed, $K_{r}^{(\ell)+}$ can be greedily embedded into any sufficiently large $k$-graph with positive minimum $\ell$-degree, since each new vertex can be chosen inside a large common neighborhood of an $\ell$-set.

Note that any two edge of $K_{r}^{(\ell)+}$ intersect on at most $\ell-1$ vertices, thus the link of any $(\ell-1)$-subset is a matching. By \cref{1vanishing}, it follows that every ordering of $V(K_{r}^{(\ell)+})$ is an $\ell$-vanishing order.
 
Now fix an $(\ell-2)$-subset $S$ of $V(K_{r}^{(\ell)})$, and let $R=V(K_{r}^{(\ell)})\setminus S$. Then any two vertices of $R$ are contained in a common edge of the link $L_{K_{r}^{(\ell)+}}(S)$. Hence this link contains a complete graph on $R$. Since $|R|=r-\ell+2>k-\ell+2$, the link of $S$ cannot be $(k-\ell+2)$-partite.
\end{proof}

By \cref{example1}, the $k$-graph $K_{r}^{(\ell+1)+}$ satisfies condition~\ref{cond1} and violates condition~\ref{cond2} in \cref{j-vanishing}.
By \cref{1vanishing}, we further know that $K_{r}^{(\ell)+}$ has no $(\ell-1)$-vanishing order when $\ell\geq2$ and $r>k$, thereby showing that \cref{general} is best possible, as promised.

For integers $r\geq k\geq 3$, the $k$-uniform \textit{tight cycle} of length $r$, denoted by $C_{r}^{(k)}$, is the $k$-graph with vertex set 
$\left\{v_1,v_2,\ldots, v_{r}\right\}$ 
and edge set 
$\left\{v_iv_{i+1}\ldots v_{i+k-1}:1\leq i\leq r\right\}$, 
where the indices are taken modulo $r$. 
The following result shows that when $r\geq 2k-1$ and $k\nmid r$, the tight cycle $C_{r}^{(k)}$ provides an example satisfying \ref{cond2} but violating \ref{cond1} in \cref{j-vanishing} as the nonexistence of a $(k-1)$-vanishing order implies the failure of condition~\ref{cond1} for all $2\le \ell \le k-2$.

\begin{observation}\label{example2}
For integers $r\geq 2k-1$ with $k\nmid r$, the $k$-uniform tight cycle $C_{r}^{(k)}$ satisfies the following.
\begin{itemize}
\item Let $2\le \ell \le k-2$. For every $(\ell-1)$-set $S$, it link is $(k-\ell+1)$-partite;
\item $C_{r}^{(k)}$ has no $(k-1)$-vanishing order.
\end{itemize}
\end{observation}

\begin{proof}
If $r\geq 2k-1$, then the link of every $(\ell-1)$-set $S$ is a tight path, and hence $(k-\ell+1)$-partite.

Now suppose that $k\nmid r$, and assume for contradiction that $C_{r}^{(k)}$ admits a $(k-1)$-vanishing order $\sigma$. We apply \cref{1vanishing} to derive a contradiction. 

Denote $V(C_{r}^{(k)})=\mathbb{Z}/r\mathbb{Z}$ and 
$E(C_{r}^{(k)})=\big\{\{i,i+1,\dots,i+k-1\} : i\in\mathbb{Z}/r\mathbb{Z}\big\}$.
For any $i\in \mathbb{Z}/r\mathbb{Z}$ and $1\leq j\leq k-1$, 
denote by $S_{ij}$ the $(k-2)$-set $\{i+1,i+2,\dots,i+k-1\}\setminus\{i+j\}$, and let $P_{ij}$ be the path on 
$\{i, i+j, i+k\}$ with edges $\{i, i+j\}$ and $\{i+j, i+k\}$. 
Note that $P_{ij}$ is a subgraph of the link of $S_{ij}$. Hence, $\sigma|_{V(P_{ij})}$ is $1$-vanishing by \cref{1vanishing}. 

Let $G$ be the union of the links of all $(k-2)$-subsets. We orient every edge $\{i,j\}\in E(G)$ as $i\rightarrow j$ whenever $\sigma(i)<\sigma(j)$. Then the corresponding oriented $P_{ij}$ cannot be a directed path; otherwise $\sigma|_{V(P_{ij})}$ would fail to be $1$-vanishing. This implies that for any $1\leq j\leq k-1$,
\begin{equation}
i\rightarrow i+j \quad \text{if and only if} \quad i+k\rightarrow i+k+j,  \label{arrow}
\end{equation}
since neither $P_{ij}$ nor $P_{i+j,k-j}$ is a directed path.

Let $d=\gcd(k,r)<k$, and assume without loss of generality that $1\rightarrow d+1$. 
Since the set 
$\{xk+1\in\mathbb{Z}/r\mathbb{Z} : 0\leq x\leq r/d-1\}$ 
consists precisely of all elements of $\mathbb{Z}/r\mathbb{Z}$ congruent to $1\pmod d$, it follows from \eqref{arrow} that $i\rightarrow i+d$ for all $i\equiv 1\pmod d$. 
Therefore, the oriented graph $G$ contains a directed cycle 
$i_1\rightarrow i_2\rightarrow \dots \rightarrow i_t\rightarrow i_1$ 
with each $i_j\equiv 1\pmod d$, which yields the contradiction
$\sigma(i_1)<\sigma(i_2)<\dots<\sigma(i_t)<\sigma(i_1)$.
\end{proof}

\section{Locally $2$-vanishing hypergraphs: Proof of \cref{2-vanishing}} \label{lowerbound} 
As discussed in~\cref{sec:pf-idea}, we prove \cref{2-vanishing} by constructing a sequence of $k$-graphs with positive minimum $2$-degree that are \emph{locally} $2$-vanishing as follows. The case $k=3$ of~\cref{2-vanishing} is covered by the results in~\cite{ding20243,reiher2018hypergraphs}. Thus we work with uniformity $k\ge 4$ in this section.

\begin{theorem} \label{vanishing2}
For all positive integers $k\geq 4$ and $m\geq 1$, there exists $\gamma>0$ such that the following hold for infinitely many $k$-graphs $H$.
\stepcounter{propcounter}
\begin{enumerate}[label = \rm({\bfseries \Alph{propcounter}\arabic{enumi}})]          \item\label{D1} $\delta_2(H)\geq\gamma |V(H)|^{k-2}$;
\item\label{D2} Any subhypergraph of $H$ on at most $m$ vertices has a $2$-vanishing order.
\end{enumerate}
\end{theorem}
\subsection{A geometric random graph}

In this subsection, we first recall a construction of an $n$-vertex random graph, which has been used in the authors' previous article \cite{ding20243} to study the codegree Tur\'{a}n density of $3$-graphs. 
Let $n$, $q$, $r$ be positive integers such that $n=qr+1$. 
The $n$-vertex random graph, denoted by $\mathcal{G}(n,r)$, is defined with vertex set 
$A=\{a_0,a_1,\ldots,a_{n-1}\}$ as follows. 
For every $0\leq i\leq n-1$ and $0\leq j\leq r-1$, let 
\[
S_{ij}=\left\{a_{i+t} : jq+1\leq t\leq (j+1)q\right\},
\] 
where the indices are taken modulo $n$. 
Let $X_0,X_1,\ldots,X_{n-1}$ be $n$ independent random variables, each of which takes a value from $\{0,1,\ldots, r-1\}$ uniformly at random. 
Then, for every $0\leq i <j \leq n-1$, $a_{i}a_{j}$ is an edge of $\mathcal{G}(n,r)$ if and only if $a_i\in S_{jX_{j}}$ and $a_j\in S_{iX_i}$. 
Note that $a_ia_j$ forms an edge of $\mathcal{G}(n,r)$ with probability $\frac{1}{r^2}$. 

As pointed out in \cite{ding20243}, a geometric way to understand $\mathcal{G}(n,r)$ is to put all the vertices of $\mathcal{G}(n,r)$ on the unit circle of the complex plane with $a_{\ell}$ corresponding to the point $e^{2\pi\ell i/n}$, and to view $S_{\ell,j}$ as an arc of length $2\pi/r$ containing all the vertices of $S_{\ell,j}$. 
Given a subset $A'=\{a_{i_1},a_{i_2},\ldots,a_{i_m}\}$ of $A$ with $0\leq i_1<i_2<\cdots<i_m\leq n-1$, we call $\sigma_c$ a \textit{cyclic ordering} of $A'$ if for some $1\leq s\leq m$, 
\[
\sigma_c(a_{i_{s+j}})=j \quad \text{for every } 1\leq j\leq m,
\]
where the indices are taken modulo $m$. 

The following lemma shows that there is a universal cyclic ordering of a subset of $A$ such that it is $1$-vanishing for any subgraph of $\mathcal{G}(n,r)$ induced on the subset.

\begin{lemma}\label{random}
 For any integer $m\leq r/2$ and any $A'\subseteq A$ of order $m$, there exists a cyclic ordering of $A'$ which is a $1$-vanishing order of any subgraph of $\mathcal{G}(n,r)$ induced on $A'$.     
\end{lemma}

\begin{proof}
Let $A'=\{a_{i_1},a_{i_2},\ldots,a_{i_m}\}$ with $0\leq i_1<i_2<\cdots<i_m\leq n-1$. 
By averaging, there exists $1\leq s\leq m$ such that the arc $\mathcal{A}$ from $a_{i_s}$ to $a_{i_{s+1}}$ (in the clockwise direction) has length at least $2\pi/m\geq 4\pi/r$, and contains no other vertices of $A'$. 
Let $\sigma_c$ be the cyclic ordering of $A'$ with $\sigma_c(a_{i_{s+j}})=j$ for $1\leq j\leq m$. 
We show that $\sigma_c$ is the desired ordering.

By the construction of $\mathcal{G}(n,r)$, for any $a_{i_j}\in A'$, all neighbors of $a_{i_j}$ lie on an arc of length $2\pi/r$. 
Since the gap $\mathcal{A}$ has length at least $4\pi/r$, this arc of neighbors cannot intersect both sides of $\mathcal{A}$. 
Therefore, with respect to $\sigma_c$, all neighbors of $a_{i_j}$ lie either entirely before $a_{i_j}$ or entirely after $a_{i_j}$.

Using this observation, we define a coloring $\varphi$ of $A'$ as follows. For $a_{i_j}\in A'$, let 
\[
\varphi(a_{i_j})=
\begin{cases}
1, & \text{if } \sigma_c(a_{i_j})<\sigma_c(a_{i_\ell}) \text{ for some } a_{i_j}a_{i_\ell}\in E(\mathcal{G}(n,r)),\\
2, & \text{otherwise}.
\end{cases}
\]
As all of $a_{i_j}$ neighbors in the induced subgraph lie entirely on one side of $a_{i_j}$ in the ordering $\sigma_c$, $\varphi$ is a proper $2$-coloring of the induced subgraph. 
Since a graph is $1$-vanishing if and only if it is bipartite, this shows that $\sigma_c$ is a $1$-vanishing order.
\end{proof}


\subsection{$(2,1^{k-2})$-type $k$-graphs}

For a $k$-graph $F$, if there is a partition $\mathcal{P}=A\cup B_1\cup\dots\cup B_{k-2}$ of the vertex set $V(F)$ such that every edge in $E(F)$ has exactly two vertices in $A$ and one vertex in $B_i$ for each $1\leq i\leq k-2$, then we say that $F$ is of \textit{$(2,1^{k-2})$-type}, and call $\mathcal{P}$ a \textit{$(2,1^{k-2})$-partition} of $F$. 
As a building block in our final construction, we first give a $(2,1^{k-2})$-type $k$-graph such that every subhypergraph of bounded order admits a $2$-vanishing order and every pair of vertices that lie in an edge has large codegree.

Let $n$, $q$, $r$ be positive integers such that $n=qr+1$, and let $A=\{a_0,a_1,\ldots,a_{n-1}\}$, $B_i=\{b_{i,1},b_{i,2},\ldots,b_{i,n}\}$ for $1\leq i\leq k-2$. We define $\mathcal{H}(n,r)$ to be the $(2,1^{k-2})$-type $k$-graph with the $(2,1^{k-2})$-partition $A\cup B_1\cup\dots\cup B_{k-2}$ and with the edge set 
\[
\left\{a_ia_jb_{1,i_1}\dots b_{k-2,i_{k-2}} : a_ia_j\in E(G_{1,i_1})\cap\dots \cap E(G_{{k-2,i_{k-2}}})\right\},
\]
where $G_{i,j}$'s are i.i.d. copies of $\mathcal{G}(n,r)$ on the vertex set $A$ for all $1 \leq i\leq k-2$ and $j\in [n]$.
In other words, each vertex $b_{i,j}$ is assigned an independent random graph $G_{i,j}$ on $A$ and when we pick one vertex from each $B_i$, the link graph of these $k-2$ vertices is the intersection of their assigned random graphs.

Given an ordering of a set and a partition of the set, we say that the ordering is a \textit{cluster ordering} respecting this partition if it can be written as the sum of some orderings of these parts. 

\begin{lemma}\label{21}
For any integer $m\leq r/2$ and any $A'\subseteq A$ of order $m$, let $\sigma$ be a cluster ordering  respecting the partition 
$A'\cup B_1\cup\dots\cup B_{k-2}$ with $\sigma|_{A'}$ being a cyclic ordering as in \cref{random}. Then $\sigma$ is a $2$-vanishing order of any subhypergraph of $\mathcal{H}(n,r)$ induced on $A'\cup B_1\cup\dots\cup B_{k-2}$. \end{lemma}
\begin{proof}
Let $F$ be a subhypergraph of $\mathcal{H}(n,r)$ induced on $A'\cup B_1\cup\dots\cup B_{k-2}$. Without loss of generality, we may assume $\sigma=\sigma_c\oplus\left(\sum_{i=1}^{k-2}\sigma_i\right)$,
where $\sigma_c$ is a cyclic ordering of $A'$ as in \cref{random} and $\sigma_i$ is an arbitrary ordering of $B_i$ for $1\leq i\leq k-2$. 
Now we aim to show that the link graph of each vertex in $F$ is $1$-vanishing under $\sigma$, then by \cref{1vanishing} (with $\ell=2$), $\sigma$ is a $2$-vanishing order of $F$.   

For each $a\in A'$, since $L_F(a)$ is a $(k-1)$-partite $(k-1)$-graph with the partition $A'\cup B_1\cup\dots \cup B_{k-2}$, we can simply color the vertices in $A'$ with $1$ and color the vertices in $B_i$ with $i+1$ for $1\leq i\leq k-2$ to get a $1$-vanishing coloring of $L_F(a)$ under $\sigma$. 

For every $1\leq t\leq k-2$ and $b\in B_{t}$, let $G_b=\{aa' : a,a'\in A'\text{~and~}d_F(aa'b)\geq 1\}$. So $G_b$ is a subgraph of $\mathcal{G}(n,r)$. By \cref{random}, there is a $1$-vanishing coloring $\varphi: A'\rightarrow \{1,2\}$ of $G_b$. Let  
\[ \varphi_b (x)= 
\left\{
\begin{aligned}
&\varphi(x), \qquad\qquad  \text{if} \ x\in A',  \\
&i+2, \qquad\qquad   \text{if} \ x\in  B_i,\ 1\leq i<t,\\ 
&i+1, \qquad\qquad   \text{if} \ x\in  B_i,\ t<i\leq k-2.
\end{aligned}
\right. 
\]
One can easily check that $\varphi_b$ is a $1$-vanishing coloring of $L_F(b)$ under $\sigma$. 
\end{proof}

Next we show that with high probability, every pair of vertices of $\mathcal{H}(n,r)$ that lie in an edge has large codegree.

\begin{lemma}\label{good}
For all integers $k\geq 4$ and $r\geq 2$, there exists $\varepsilon>0$ such that for infinitely many integers $n$, the $(2,1^{k-2})$-type $k$-graph $\mathcal{H}(n,r)$ satisfies the following with high probability.
\stepcounter{propcounter}
\begin{enumerate}
[label = \rm({\bfseries \Alph{propcounter}\arabic{enumi}})]          
\item\label{C1} $d(aa')\geq\varepsilon n^{k-2}$ for any $a,a'\in A$;
\item\label{C2} $d(ab)\geq\varepsilon n^{k-2}$ for any $a\in A$, $b\in B_i$ and $1\leq i\leq k-2$;
\item\label{C3}  $d(bb')\geq\varepsilon n^{k-2}$ for any $b\in B_i$, $b'\in B_j$ and $1\leq i<j\leq k-2$.
\end{enumerate}
\end{lemma}
\begin{proof}
Set $\varepsilon=\frac{1}{(6r^3)^{k-2}} $, and let $n=qr+1$ with $n\gg r$.

\medskip

\noindent{\bf Verifying~\ref{C1}.} 
For any $a_i,a_j\in A$, let 
$$Y_s=|\{t\in[n]:a_ia_j\in E(G_{s,t})\}|,$$ 
which is the number of random graphs assigned to the vertices in $B_s$ that include $a_ia_j$ as an edge. By the definition of $\mathcal{H}(n,r)$, 
$\mathbb{P}\left(a_ia_j\in E(G_{s,t})\right)=\frac{1}{r^2}$,
$Y_s\sim {\rm Bin}(n,\frac{1}{r^2})$ and $\mathbb{E}[Y_s]=n/r^2$. Therefore, it follows from Chernoff's inequality that $\mathbb{P}\left(Y_s\leq \frac{n}{2r^2}\right)\leq e^{-\frac{n}{8r^2}}.$ Noting that 
$d(a_ia_j)=Y_1\dots Y_{k-2}$,
we obtain  $\mathbb{P}\left(d(a_ia_j)\leq \frac{n^{k-2}}{(2r^2)^{k-2}}\right)\leq (k-2)e^{-\frac{n}{8r^2}}.$ Taking a union bound over all pairs of $A$, the probability that $d(a_ia_j)\geq  \frac{n^{k-2}}{(2r^2)^{k-2}}$ for every pair of vertices $a_i,a_j$ is at least $1-(k-2)\binom{n}{2}e^{-\frac{n}{8r^2}}$.

\medskip

\noindent{\bf Verifying~\ref{C2}.} We first estimate the probability that $d(a_0b_{1,1})<\varepsilon n^{k-2}$.  Let $S$ be the random set chosen for $a_0$ in $G_{1,1}$, so $|S|=q=(n-1)/r$, and let $S'$ be the subset of $S$ consisting of all those $a_j$ such that $a_0$ is in the random set chosen for $a_j$ in $G_{1,1}$.  Then $|S'|\sim {\rm Bin}(q,1/r)$, $\mathbb E(|S'|)=(n-1)/r^2$, and for  sufficiently large $n$, Chernoff's inequality gives
$\mathbb{P}\left(|S'|\leq \frac{n}{2r^2}\right)\leq e^{-\frac{n}{36r^2}}.$

Now fix a vertex $a_s\in S'$, and let
$B'_i=\{b_{i,j}\in B_i:a_sa_0\in E(G_{i,j})\}$
for $2\leq i\leq k-2$.  Then $|B'_i|\sim {\rm Bin}(n,1/r^2)$ and
$\mathbb{P}\left(|B'_i|\leq \frac{n}{2r^2}\right)\leq e^{-\frac{n}{8r^2}}.$
Since $d(a_0a_sb_{1,1})=|B'_2|\cdots |B'_{k-2}|$, we have
$$\mathbb{P}\left(d(a_0a_sb_{1,1})\leq \frac{n^{k-3}}{(2r^2)^{k-3}}\right)
\leq (k-3)e^{-\frac{n}{8r^2}}.$$
Therefore,
\begin{align*}
\mathbb{P}\left(d(a_0b_{1,1})\leq
\frac{n^{k-3}}{(2r^2)^{k-3}}\frac{n}{2r^2}\right)
&\leq \mathbb{P}\left(|S'|\leq \frac{n}{2r^2}\right)
+\mathbb{P}\left(d(a_0a_sb_{1,1})\leq \frac{n^{k-3}}{(2r^2)^{k-3}}
\text{ for some }a_s\in S'\right)\\
&\leq e^{-\frac{n}{36r^2}}+(k-3)n e^{-\frac{n}{8r^2}}
\leq e^{-c_2 n}
\end{align*}
for some constant $c_2=c_2(k,r)>0$ and all sufficiently large $n$.  By symmetry, for every
$b_{i,j}\in B_i$ and $a_\ell\in A$,
$\mathbb{P}\left(d(b_{i,j}a_\ell)\leq \frac{n^{k-2}}{(2r^2)^{k-2}}\right)
\le e^{-c_2 n}.$
Thus \ref{C2} holds for all $(k-2)n^2$ such pairs with probability at least
$1-(k-2)n^2e^{-c_2n}.$

\medskip

\noindent{\bf Verifying~\ref{C3}.} 
By symmetry, it suffices to estimate the probability that $d(b_{1,1}b_{2,1})<\varepsilon n^{k-2}$. To eliminate the dependence between the edges in the random graph and use a tail inequality, we consider a bipartite subgraph. Let $S^{-}=\{a_{-1},\dots,a_{-n/3r}\}$ and $S^{+}=\{a_{1},\dots,a_{n/3r}\}$. Denote by $G'$ the induced bipartite subgraph of $G_{1,1}\cap G_{2,1}$ on $S^{-}\cup S^{+}$. 
Let 
\[
W^{+}=\{a_j\in S^{+}:\text{the random set chosen for } a_j \text{ in } G_{1,1} \text{ and } G_{2,1} \text{ are both } S_{j,r-1}\}
\]
and
\[
W^{-}=\{a_j\in S^{-}:\text{the random set chosen for } a_j \text{ in } G_{1,1} \text{ and } G_{2,1} \text{ are both } S_{j,0}\}.
\]
Then $|W^{+}|, |W^{-}|\sim {\rm Bin}(\frac{n}{3r},\frac{1}{r^2})$, and every pair in
$W^+\times W^-$ forms an edge of $G'$.  Thus
$|E(G')|\ge |W^+||W^-|.$
Therefore,
\[
\mathbb{P}\left( |E(G')|\leq  \frac{n^2}{36r^6} \right)
\leq \mathbb{P}\left( |W^{+}|\leq  \frac{n}{6r^3} \text{~or~} |W^{-}|\leq  \frac{n}{6r^3}\right)
\leq 2e^{-\frac{n}{24r^3}}.
\]
For each possible edge $e$ between $S^-$ and $S^+$ and each $3\leq i\leq k-2$, let
\[
W^i_e=\{b_{i,j}:e\in E(G_{i,j})\}.
\]
Then $|W^i_e|\sim {\rm Bin}(n,1/r^2)$ and
$\mathbb{P}\left( |W^i_e|\leq  \frac{n}{2r^2} \right)\leq e^{-\frac{n}{8r^2}}.$
If $|E(G')|>n^2/(36r^6)$ and $|W^i_e|>n/(2r^2)$ for every $e\in G'$ and every
$3\le i\le k-2$, then
\[
d(b_{1,1}b_{2,1})
\ge |E(G')|\left(\frac{n}{2r^2}\right)^{k-4}
\ge \frac{n^{k-2}}{(6r^3)^{k-2}}.
\]
Hence, using the union bound over all at most $n^2$ possible choices of $e$,
\[
\mathbb{P}\left( d(b_{1,1}b_{2,1})\leq  \frac{n^{k-2}}{(6r^3)^{k-2}} \right)
\leq 2e^{-\frac{n}{24r^3}}+(k-4)n^2e^{-\frac{n}{8r^2}}
\leq e^{-\frac{n}{25r^3}}
\]
for sufficiently large $n$.
By symmetry, we know that 
$\mathbb{P}\left(d(b_{i_1,j_1}b_{i_2,j_2})\leq \frac{n^{k-2}}{(6r^3)^{k-2}}\right)\leq  e^{-\frac{n}{25r^3}}$
for every $b_{i_1,j_1}\in B_{i_1}$ and $b_{i_2,j_2}\in B_{i_2}$ with $i_1\neq i_2$. 
Taking a union bound over all $\binom{k-2}{2} n^{2}$ such pairs,  \ref{C3} holds with probability at least 
$1-\binom{k-2}{2} n^{2} e^{-\frac{n}{25r^3}}$.
\end{proof}

\subsection{The final hypergraph: Proof of~\cref{vanishing2}}

To complete our final construction, we glue several $(2,1^{k-2})$-type $k$-graphs in a suitable way to get a $k$-graph which has positive minimum $2$-degree while maintaining the locally $2$-vanishing property. For this purpose, we utilize the following classic result in combinatorial design combined with a random sparsification approach.

\begin{theorem}[\cite{wilson1975existence}]\label{design}
For every integer $k\geq 3$, there is $N>0$ such that for every $n>N$ with $n -1\equiv 0 \pmod{k-1}$ and $(n -1)n\equiv 0  \pmod{(k-1)k}$, there is a $k$-graph on $[n]$ such that every pair of vertices is contained in exactly one edge.
\end{theorem}

Now we are ready to give our final construction.

\begin{proof}[Proof of~\cref{vanishing2}]
By \cref{design}, for some large enough integer $N$, there is a $(k-1)$-graph $\mathcal{J}$ on $[N]$ such that every pair of vertices is contained in exactly one edge.  
Let $\varepsilon>0$ be chosen from \cref{good} for $k$ and $r=4m$, and set $n\gg N, m, k, 1/\varepsilon$. 
Then by \cref{good}, there is a $(2,1^{k-2})$-type $k$-graph $\mathcal{H}$ with each part of size $n$, which satisfies \ref{C1}, \ref{C2}, \ref{C3} and possesses the property of $\mathcal{H}(n,r)$ given in \cref{21}.

Next we construct a $k$-graph $\widehat{H}$ on $\cup_{i=1}^N V_i$ by placing a copy of $\mathcal{H}$ on $\cup_{i\in J}V_i$ for every $J\in E(\mathcal{J})$, where $V_1,V_2,\dots,V_N$ are $N$ disjoint vertex sets each of size $n$. 

\begin{claim}\label{claim:surjphi}
There exists a surjection $\varphi:E(\mathcal J)\to [N]$ such that $\varphi(J)\in J$ for every $J\in E(\mathcal J)$.
\end{claim}

\begin{poc}
Consider the bipartite incidence graph $B$ with left part $[N]$ and right part $E(\mathcal J)$, where $i\in[N]$ is adjacent to $J\in E(\mathcal J)$ if and only if $i\in J$.
It suffices to show that $B$ contains a matching saturating the left part $[N]$. Note that in the design $\mathcal J$, each vertex $i\in[N]$ lies in exactly $r=\frac{N-1}{k-2}$ edges. 

Let $X\subseteq [N]$ be nonempty. 
Counting incidences between $X$ and its neighborhood $N_B(X)$, we have
$r|X| \;=\; e_B(X,N_B(X)) \;\le\; (k-1)\,|N_B(X)|.$
Therefore, when $N$ is large enough,
\[
|N_B(X)| \;\ge\; \frac{r}{k-1}|X| \;=\; \frac{N-1}{(k-2)(k-1)}\,|X|\ge |X|.
\]
By Hall's theorem, $B$ has a matching saturating $[N]$.
\end{poc}

Let $\varphi$ be a map guaranteed above. For every $J\in E(\mathcal{J})$, let $H_J$ be a copy of the $(2,1^{k-2})$-type $k$-graph $\mathcal{H}$ such that each edge of $H_J$ has two vertices in $V_{\varphi(J)}$ and one vertex in $V_i$ for each $i\in J\setminus\{\varphi(J)\}$. 
Furthermore, we require that for different $J,J'\in E(\mathcal{J})$ with $\varphi(J)=\varphi(J')$, the two copies $H_J$ and $H_{J'}$ agree on $V_{\varphi(J)}$. Let 
$E(\widehat{H})=\cup_{J\in E(\mathcal{J})}E(H_J).$

We claim that $\delta_2(\widehat H)\ge \varepsilon n^{k-2}$. Indeed, fix any two vertices $x,y\in V(\widehat H)$, and write $x\in V_i$ and $y\in V_j$. If $i\neq j$, then by \cref{design} there is a unique $J\in E(\mathcal J)$ with $\{i,j\}\subseteq J$. Consider the copy $H_J$ placed on $\bigcup_{t\in J}V_t$. Then $d_{\widehat H}(xy)\ge d_{H_J}(xy)\ge \varepsilon n^{k-2}$ due to \cref{good}\ref{C2} and \ref{C3}. If $i=j$, then by the surjectivity of $\varphi$ there exists $J\in E(\mathcal J)$ with $\varphi(J)=i$. In the copy $H_J$, both $x$ and $y$ lie in the $A$-part $V_i$, hence \cref{good}\ref{C1} yields $d_{\widehat H}(xy)\ge d_{H_J}(xy)\ge \varepsilon n^{k-2}$.

However, $\widehat{H}$ is not necessarily locally $2$-vanishing. The problem here is that the link graph of a vertex of $\widehat{H}$ may consist of intersecting edges from different $(2,1)$-type $k$-graphs.
To separate these link graphs, we further perform a random sparsification of the edges of $\widehat{H}$.

Let     
\[
\phi:\bigcup_{i=1}^{N}\binom{V_i}{2}\to E(\mathcal{J})
\]
be a random mapping such that for each $1\leq i\leq N$ and any $uv\in \binom{V_i}{2}$, $\phi(uv)$ is chosen independently and uniformly at random from the set 
$$E_i=\{J\in E(\mathcal{J}):\varphi(J)=i\}.$$ 
For every $J\in E(\mathcal{J})$, let $H'_J$ be the subhypergraph of $H_J$ with 
\[ 
E(H'_J)=\{e\in E(H_J) : \phi(e\cap V_{\varphi(J)})=J\}.
\]
Further, let $H$ be the subhypergraph of $\widehat{H}$ with 
$$E(H)=\cup_{J\in E(\mathcal{J})}E(H'_J).$$

\begin{claim}\label{claim:packing}
For any $v\in V(H)$ and two distinct edges $J, J'$ in $\mathcal{J}$, the links $L_{H'_{J}}(v)$ and $L_{H'_{J'}}(v)$ are vertex-disjoint. 
\end{claim}
\begin{poc}
Fix $v\in V_i$ and let $J\neq J'$ be edges of $\mathcal J$. 
Suppose for a contradiction that some vertex $w\in V_t$ lies in both 
$V(L_{H'_J}(v))$ and $V(L_{H'_{J'}}(v))$.

If $t\neq i$, then there are edges of $H'_J$ and $H'_{J'}$ containing $\{v,w\}$,
so $\{i,t\}\subseteq J\cap J'$, contradicting the fact that $\C J$ is design.

If $t=i$, then $\{v,w\}$ can be contained in an edge of $H_J$ only when
$\varphi(J)=i$.
Thus $\varphi(J)=\varphi(J')=i$, and any edge $e\in H'_J$ containing $\{v,w\}$
satisfies $\phi(\{v,w\})=J$ by definition of $H'_J$. Similarly, an edge of $H'_{J'}$
containing $\{v,w\}$ implies $\phi(\{v,w\})=J'$. Since $\phi$ is a function,
we obtain $J=J'$, again a contradiction.
\end{poc}

Setting $\gamma=\frac{\varepsilon^2}{32N^{k}}$, we next show that for $H$,~\ref{D1} is satisfied with high probability and~\ref{D2} is always true, and thereby finish the proof. 

\medskip

\noindent{\bf Verifying~\ref{D1}.} 
Let $u\in V_i$ and $v\in V_j$ be two distinct vertices of $H$. If $i=j$, then $J=\phi(uv)$ is defined and by definition
$$d_{H}(uv)=d_{H'_{J}}(uv)=d_{H_{J}}(uv)\geq \varepsilon n^{k-2}\geq \gamma |V(H)|^{k-2}.$$
If $i\neq j$, by the construction of $\mathcal{J}$, there is a unique edge $J\in E(\mathcal{J})$ containing $\{i,j\}$. Suppose that $\varphi(J)\in\{i, j\}$. Without loss of generality, we may assume that $\varphi(J)=i$.
Let 
\[
U=\{w\in V_{i} : d_{H_J}(uvw)\geq \varepsilon n^{k-3}/2\} 
\text{~and~} 
W=\{w\in U : \phi(uw)=J\}.
\]
Then $|U|\geq \varepsilon n/2$ as $d_{H_J}(uv)\geq \varepsilon n^{k-2}$ and 
$|W|\sim {\rm Bin}(|U|,1/|E_{i}|)$.
As $|E_i|\le |E(\mathcal{J})|<N^2$, by Chernoff's inequality, we conclude that 
\[
\mathbb{P}\left(|W|\leq \frac{\varepsilon n}{4N^2}\right)
\leq 
\mathbb{P}\left(|W|\leq \frac{\varepsilon n}{4|E_i|}\right)
\leq e^{-\frac{\varepsilon n}{16N^2}}.
\]
Therefore, with probability at least $1-e^{-\frac{\varepsilon n}{16N^2}}$, 
\[
d_{H}(uv)
=d_{H'_J}(uv)
\geq \frac{\varepsilon n^{k-3}}{2}|W|
\geq \frac{\varepsilon^2 n^{k-2}}{8N^2}
\geq \gamma |V(H)|^{k-2}.
\]

Now suppose that $\varphi(J)=\ell\notin \{i, j\}$. Let 
\[
X=\left\{xy\in\binom{V_\ell}{2} : d_{H_J}(xyuv)\geq \varepsilon n^{k-4}/2\right\}
\text{~and~}
Y=\{xy\in X : \phi(xy)=J\}.
\]
Then $|X|\geq \varepsilon n^2/4$ as $d_{H_J}(uv)\geq \varepsilon n^{k-2}$ and 
$|Y|\sim {\rm Bin}(|X|,1/|E_\ell|)$.
Therefore, by Chernoff's inequality,
\[
\mathbb{P}\left(|Y|\leq \frac{\varepsilon n^2}{8N^2}\right)
\leq \mathbb{P}\left(|Y|\leq \frac{|X|}{2|E_\ell|}\right)
\leq e^{-\frac{\varepsilon n^2}{32N^2}},
\]
meaning that with probability at least $1-e^{-\frac{\varepsilon n^2}{32N^2}}$, 
\[
d_{H}(uv)
=d_{H'_J}(uv)
\geq \frac{\varepsilon n^{k-4}}{2}|Y|
\geq \frac{\varepsilon^2 n^{k-2}}{16N^2}
\geq \gamma |V(H)|^{k-2}.
\]

Combining all of the above cases and taking a union bound over all pairs of $V(H)$, we conclude that with probability at least $1-\binom{|V(H)|}{2}e^{-\frac{\varepsilon n}{16N^2}}$, $\delta_2(H)\geq \gamma |V(H)|^{k-2}$.

\medskip

\noindent{\bf Verifying~\ref{D2}.}
Let $F$ be a subhypergraph of $H$ with $|V(F)|\leq m$. For each $1\leq i\leq N$, let $A_i=V(F)\cap V_i$ and $\sigma_i$ be a cyclic ordering of $A_i$ as in \cref{random}. Then we get an ordering $\sigma=\sum_{i=1}^{N}\sigma_i$ of $V(F)$. In other words, $\sigma$ is a cluster ordering respecting the partition $\cup_{i=1}^{N}A_i$ with $\sigma|_{A_i}=\sigma_i$ being a cyclic ordering of $A_i$ as in \cref{random} for $1\leq i\leq N$. 

For any $1\leq i\leq N$ and $v\in A_i$, note that $$L_{H}(v)=\bigcup_{J\in E(\mathcal{J}),\,i\in J}L_{H'_J}(v)$$ 
and by~\cref{claim:packing} these link hypergraphs are vertex-disjoint. By \cref{1vanishing},
it suffices to show that $\sigma$ is a $1$-vanishing order of each $L_{H'_J}(v)$. 

Since $H'_J$ is a subhypergraph of $\mathcal{H}(n,r)$ and $|V(F)\cap V(H'_J)|\le m\le r/2$, the restriction of $\sigma$ to $V(H'_J)\cap V(F)$ is a cluster ordering satisfying the conditions of \cref{21}. Hence it is a $2$-vanishing order of $H'_J[V(F)]$, which implies by \cref{1vanishing} that $\sigma$ is a $1$-vanishing order of $L_{H'_J}(v)$. This completes the proof.
\end{proof}

\begin{proof}[Proof of \cref{2-vanishing}]
The case $k=3$ follows from \cref{zeroco}.  Let $k\ge4$, and suppose contrapositively that
$F$ has no $2$-vanishing order.  Put $m=|V(F)|$.  By \cref{vanishing2}, there are infinitely
many $k$-graphs $H$ such that
$\delta_2(H)\ge \gamma |V(H)|^{k-2}$
for some $\gamma>0$, and every subhypergraph of $H$ on at most $m$ vertices has a
$2$-vanishing order.  Such an $H$ is $F$-free: otherwise a copy of $F$ in $H$ would be a
subhypergraph on at most $m$ vertices and hence would have a $2$-vanishing order, contradicting
our assumption on $F$.
Therefore
\[
\pi_2(F)\ge \limsup_{|V(H)|\to\infty}
\frac{\gamma |V(H)|^{k-2}}{\binom{|V(H)|-2}{k-2}}>0.
\]
This proves the contrapositive, and hence \cref{2-vanishing}.
\end{proof}

\section{Suspensions: Proof of \cref{suspension}}\label{suspenandaccu}

In this section, we first prove \cref{suspension} by a random construction, and then deduce \cref{j-vanishing} from \cref{suspension} together with a straightforward random construction.

\subsection{Probabilistic tools}

We need the following three probabilistic inequalities in our proofs.

\begin{lemma}[Bounded differences for random bijections, \cite{Mcdiarmid2002}]\label{lemperm-bd}
Let $\rho$ be a uniformly random bijection between two $m$-element sets, and
let $Z=f(\rho)$. Suppose that transposing two values of $\rho$ changes $f$ by at most $L$. 
Then for every $t>0$, we have
$$\mathbb{P}\bigl(Z<\mathbb E(Z)-t\bigr)\le\exp\left(-\frac{t^2}{2mL^2}\right).$$
\end{lemma}

\begin{lemma}[McDiarmid's inequality, \cite{McDiarmid1989}]\label{lem:mcdiarmid}
Let $X_1,\ldots,X_m$ be independent random variables, where
$X_i$ takes values in a set $\Omega_i$. Let $f:\Omega_1\times\cdots\times\Omega_m\to \mathbb R$
be a function satisfying the bounded-differences condition
$$|f(x_1,\ldots,x_m)-f(x_1,\ldots,x_{i-1},x_i',x_{i+1},\ldots,x_m)|\le c_i$$
for every $1\le i\le m$, every $x_j\in\Omega_j$, and every
$x_i'\in\Omega_i$. If
$Z=f(X_1,\ldots,X_m)$,
then for every $t>0$, we have
$$\mathbb{P}(Z\le \mathbb E (Z)-t)\le
\exp\left(-\frac{2t^2}{\sum_{i=1}^m c_i^2}\right).$$
\end{lemma}

\begin{theorem}[Azuma's inequality]\label{azuma}
 Let $X_0,X_1,\ldots,X_N$ be a martingale with $|X_{i}-X_{i-1}|\leq c_i$ for all $1\leq i\leq N$. Then for any $\lambda>0$, we have 
 $$
 \mathbb{P}\left(X_N<X_0-\lambda\right)\leq 
 \exp\left(-\frac{\lambda^2}{2\sum_{i=1}^Nc_i^2}\right).
 $$
\end{theorem}

We first apply Lemmas \ref{lemperm-bd} and \ref{lem:mcdiarmid} to give two elementary random relabelling estimates, which will be used in the subsequent proof of \cref{suspension}. 

\begin{lemma}\label{tool}
Fix integers $2\le \ell<k$ and real numbers $0<\alpha,\beta\le 1$.  There is a constant
$c=c(k,\ell,\alpha,\beta)>0$ such that the following two statements hold for all sufficiently large $m$.

\stepcounter{propcounter}
\begin{enumerate}[label = \rm({\bfseries \Alph{propcounter}\arabic{enumi}})] 
\item\label{S1} Let $U$ and $W$ be two sets of size $m$.  Let
$\mathcal A\subseteq U^{(k-\ell)}$ and $\mathcal B\subseteq W^{(k-\ell)}$ satisfy that
$|\mathcal A|\ge \alpha\binom{m}{k-\ell}$ and $|\mathcal B|\ge \beta\binom{m}{k-\ell}.$
If $\rho:U\to W$ is a uniformly random bijection, then $$ \mathbb{P}\left( |\mathcal A\cap \rho^{-1}(\mathcal B)| <\frac{\alpha\beta}{2}\binom{m}{k-\ell} \right) \le e^{-cm},$$
where $\rho^{-1}(\mathcal B):=\left\{\{\rho^{-1}(x_1),\dots,\rho^{-1}(x_{k-\ell})\} : \{x_1,\dots,x_{k-\ell}\}\in\mathcal{B}\right\}$.

\item\label{S2} Let $S$ and $U$ be two disjoint sets with $|S|=\ell$ and $|U|=m$, and $H$ be a $(k-1)$-graph on $m+\ell-1$ vertices with
$|E(H)|\ge \alpha\binom{m+\ell-1}{k-1}.$
For each $y\in U$, independently choose a uniformly random bijection
$\sigma_y:(S\cup U)\setminus\{y\}\to V(H)$, and let 
$H_y$ be the $(k-1)$-graph on $(S\cup U)\setminus\{y\}$ with
$E(H_y)=\sigma^{-1}_y(E(H))$.
Let $\mathcal C\subseteq U^{(k-\ell)}$ be any fixed family satisfying that
$|\mathcal C|\ge \beta\binom{m}{k-\ell}$,
and $Y$ be the number of sets $R\in\mathcal C$ such that
$S\cup (R\setminus\{y\})\in E(H_y)$ for every $y\in R$.
Then $\mathbb{P}\left( Y<\frac12\alpha^{k-\ell}\beta\binom{m}{k-\ell} \right)\le e^{-cm}.$
\end{enumerate}
\end{lemma}
\begin{proof}
For the first statement, set
$Z=|\mathcal A\cap \rho^{-1}(\mathcal B)|.$
Then
$\mathbb{E} (Z) =\frac{|\mathcal A||\mathcal B|}{\binom{m}{k-\ell}}\ge\alpha\beta\binom{m}{k-\ell}.$
If two values of $\rho$ are transposed, then only those $(k-\ell)$-sets containing exactly one of the two affected points can change their contribution to $Z$.  Hence the change in $Z$ is at most
$2\binom{m-1}{k-\ell-1}.$
Applying \cref{lemperm-bd} with $t=\frac12\alpha\beta\binom{m}{k-\ell}$ gives $\mathbb{P}\left(Z<\frac{\alpha\beta}{2}\binom{m}{k-\ell}\right)\le e^{-cm}$ for some $c=c(k,\ell,\alpha,\beta)>0$.

For the second statement, fix $R\in\mathcal C$.  For each $y\in R$, the set $S\cup(R\setminus\{y\})$ is a fixed $(k-1)$-set inside $(S\cup U)\setminus\{y\}$.  Since $H_y$ is a uniformly random copy of $H$,
$$\mathbb{P}\bigl(S\cup(R\setminus\{y\})\in E(H_y)\bigr)=\frac{|E(H)|}{\binom{m+\ell-1}{k-1}}\ge \alpha.$$
The events for distinct $y\in R$ depend on independent random bijections, so
$$\mathbb{P}\bigl(R\text{ is counted by }Y\bigr)\ge \alpha^{k-\ell}.
$$
Therefore $\mathbb{E} (Y)\ge \alpha^{k-\ell}|\mathcal C|\ge \alpha^{k-\ell}\beta\binom{m}{k-\ell}.$
Now view $Y$ as a function of the independent random bijections $(\sigma_y)_{y\in U}$.  Resampling one bijection $\sigma_y$ can affect only those sets $R\in\mathcal C$ containing $y$, and there are at most $\binom{m-1}{k-\ell-1}$
such sets. 
Applying \cref{lem:mcdiarmid} with $t=\frac12\alpha^{k-\ell}\beta\binom{m}{k-\ell}$
yields that $\mathbb{P}\left(Y<\frac12\alpha^{k-\ell}\beta\binom{m}{k-\ell}\right)\le e^{-cm}$ for some $c=c(k,\ell,\alpha,\beta)>0$.
\end{proof}

\subsection{Proof of \cref{suspension}}

For any $(k-1)$-graph $F$, it is easy to see that $\pi_{\ell}(\mathcal  S_F)\le \pi_{\ell-1}(F)$ for all $2\leq \ell<k$. Consequently, $\pi_{\ell-1}(F)=0$ implies $\pi_{\ell}(\mathcal{S}_F)=0$. For the other direction, let $F$ be a $(k-1)$-graph with $\pi_{\ell-1}(F)=\alpha>0$. 
Then for every sufficiently large integer $N$, there exists an $F$-free $(k-1)$-graph $H$ on $N-1$ vertices with
$\delta_{\ell-1}(H) \ge \frac{\alpha}{2}\binom{N-\ell}{k-\ell}$.
We next show that $\pi_\ell(\mathcal{S}_F) > 0$ by constructing an $\mathcal{S}_F$-free $k$-graph $\mathcal{H}$ on $N$ vertices with a positive minimum $\ell$-degree.

For each $x \in [N]$, independently choose a uniformly random bijection $\sigma_x \colon [N] \setminus \{x\} \to V(H)$, and let $H_x$ be the $(k-1)$-graph on $[N] \setminus \{x\}$ with $E(H_x)=\sigma_x^{-1}(E(H))$.
Thus a $(k-1)$-set $A \subseteq [N] \setminus \{x\}$ is an edge of $H_x$ if and only if $\sigma_x(A):=\{\sigma_x(y) : y\in A\}$ is an edge of $H$.
Now define the $k$-graph $\mathcal{H}$ by letting a $k$-set $e \subseteq [N]$ be an edge of $\mathcal{H}$ if and only if
\begin{equation}\label{defedgeofH}
e \setminus \{x\} \in E(H_x) \qquad\text{for every } x \in e.
\end{equation}
Clearly, $L_{\mathcal{H}}(x) \subseteq H_x$ holds for every $x \in [N]$. 
Since $H_x$ is $F$-free, it follows that $\mathcal{H}$ is $\mathcal{S}_F$-free. 

We next show that, with high probability, $\mathcal{H}$ has a positive minimum $\ell$-degree.  
Fix an $\ell$-set $S \subseteq [N]$, and let $U = [N] \setminus S$ and $m=N-\ell$.
For each $x \in S$, define
$$
\mathcal{A}_x = \left\{R \in U^{(k-\ell)} : (S \setminus \{x\}) \cup R \in E(H_x)\right\}.
$$
Clearly, $L_{\mathcal{H}}(S)\subseteq \bigcap_{x \in S} \mathcal{A}_x$, and 
$|\mathcal{A}_x|\geq \delta_{\ell-1}(H_x) \ge \frac{\alpha}{2}\binom{m}{k-\ell}$ for every $x \in S$.
The following claim shows that with high probability $\bigcap_{x \in S} \mathcal{A}_x$ also has a large cardinality. 

\begin{claim}\label{claiminduction}
There is a constant $c_1 = c_1(\alpha, k, \ell) > 0$ such that 
$$
\left|\bigcap_{x \in S} \mathcal{A}_x\right| \ge \left(\frac{\alpha}{4}\right)^\ell \binom{m}{k-\ell}
$$
holds with probability at least $1 - e^{-c_1N}$.
\end{claim}

\begin{poc}
Suppose that $S = \{x_1, x_2, \dots, x_\ell\}$, and for $1\leq i \leq\ell$, let $B_i$ be the event that
$\Bigl|\bigcap_{j=1}^{i} \mathcal{A}_{x_j}\Bigr| \ge \left(\frac{\alpha}{4}\right)^{i} \binom{m}{k-\ell}.$
Let $d_i=c(k,\ell,(\frac{\alpha}{4})^i,\frac{\alpha}{2})$ given by \cref{tool} for $1\leq i\leq \ell-1$ and $c'=\min\{d_1,\dots,d_{\ell-1}\}$.
We next prove that $\mathbb{P}(B_i)\geq 1-ie^{-c'm}$ by induction on $i$, and therefore our claim follows from that 
$\mathbb{P}(B_\ell)\geq 1-\ell e^{-c'm}\geq 1-e^{-c_1 N}$ for some constant $c_1$.

Since $|\mathcal{A}_{x}| \ge \frac{\alpha}{2}\binom{m}{k-\ell}$ for any $x\in S$, we have $\mathbb{P}(B_1)=1$. Now suppose that $i>1$ and $\mathbb{P}(B_{i-1})\geq 1-(i-1)e^{-c'm}$. 
For an $(\ell-1)$-subset $T$ of $V(H)$, let $\Psi_T$ be the collection of all bijections $\sigma: [N]\setminus\{x_i\}\rightarrow V(H)$ satisfying that $\sigma(S\setminus\{x_i\})=T$.
To estimate $\mathbb{P}(B_i)$, we first give a bound of $\mathbb{P}\left(B_i \mid B_{i-1}\cap \{\sigma_{x_i}\in\Psi_T\}\right)$. 

Let $\rho \colon U \to V(H) \setminus T$ be a uniformly random bijection.
Set $\mathcal{B}_T=L_H(T)$ and $\mathcal{A}= \bigcap_{j=1}^{i-1} \mathcal{A}_{x_j}\subseteq U^{(k-\ell)}$. 
Note that $\mathcal{A}_{x_i}=\rho^{-1}(\mathcal{B}_T)$ and $|\mathcal{A}|\geq \left(\frac{\alpha}{4}\right)^{i-1} \binom{m}{k-\ell}$ conditional on $B_{i-1}\cap \{\sigma_{x_i}\in\Psi_T\}$.
Then we have 
\begin{align}
\mathbb{P}\left(B_i \mid B_{i-1}\cap \{\sigma_{x_i}\in\Psi_T\}\right)
&=\mathbb{P}\left(|\mathcal{A}\cap \mathcal{A}_{x_i}|\geq\left(\frac{\alpha}{4}\right)^{i} \binom{m}{k-\ell} \mid B_{i-1}\cap \{\sigma_{x_i}\in\Psi_T\}\right)\nonumber\\
&=\mathbb{P}\left(|\mathcal{A}\cap \rho^{-1}(\mathcal{B}_T)|\geq\left(\frac{\alpha}{4}\right)^{i} \binom{m}{k-\ell} \mid B_{i-1}\right)\nonumber\\
&\geq 1-e^{-d_{i-1}m},\nonumber
\end{align}
where the last inequality follows from \ref{S1} and that $|\mathcal{B}_T|\geq \frac{\alpha}{2}\binom{m}{k-\ell}$.
Therefore,  
\begin{align}
\mathbb{P}(B_i)&\geq \mathbb{P}(B_i\cap B_{i-1})\nonumber\\
&=\sum_{T\in V(H)^{(\ell-1)}}\mathbb{P}\left(B_i\cap B_{i-1}\cap \{\sigma_{x_i}\in\Psi_T\}\right)\nonumber\\
&=\sum_{T\in V(H)^{(\ell-1)}}\mathbb{P}\left(B_i\mid B_{i-1}\cap \{\sigma_{x_i}\in\Psi_T\}\right)\mathbb{P}\left(B_{i-1}\cap \{\sigma_{x_i}\in\Psi_T\}\right)\nonumber\\
&=\sum_{T\in V(H)^{(\ell-1)}}\mathbb{P}\left(B_i\mid B_{i-1}\cap \{\sigma_{x_i}\in\Psi_T\}\right)\mathbb{P}\left(B_{i-1}\right)\mathbb{P}\left(\sigma_{x_i}\in\Psi_T\right)\nonumber\\
&\geq (1-e^{-d_{i-1}m})\left(1-(i-1)e^{-c'm}\right)\left(\sum_{T\in V(H)^{(\ell-1)}}\mathbb{P}\left(\sigma_{x_i}\in\Psi_T\right)\right)\nonumber\\
&\geq 1-ie^{-c'm}.\nonumber\qedhere
\end{align}
\end{poc}

For every $R \in \bigcap_{x \in S} \mathcal{A}_x$, all conditions in (\ref{defedgeofH}) corresponding to vertices $x \in S$ are already satisfied for the $k$-set $S \cup R$.  To make $S \cup R$ an edge of $\mathcal{H}$, it remains to require that  
$S \cup (R \setminus \{y\}) \in E(H_y)$ for every $y \in R$. 
Note that the random copies $\{H_y : y \in U\}$ are independent conditional on the event $B_\ell$. 
By double-counting pairs $(T,e)$ with $T\in V(G)^{(\ell-1)}$, $e\in E(G)$, and $T\subseteq e$, we also have
$e(G)\ge\frac \alpha2\binom{m+\ell-1}{k-1}.$
Therefore, following from \ref{S2}, for some constant $c_2>0$ 
$$\mathbb{P}\left(d_{\mathcal{H}}(S)<\frac{\alpha^k}{2^{k+\ell+1}}\binom{m}{k-\ell} \middle| B_\ell \right)\le e^{-c_2N}.$$  
By \cref{claiminduction}, we conclude that for some constant $c_3>0$,
$
d_{\mathcal{H}}(S)\ge \frac{\alpha^k}{2^{k+\ell+1}}\binom{m}{k-\ell}
$
holds with probability at least $1-e^{-c_3N}$.

Finally, taking a union bound over all $\ell$-sets $S\subseteq [N]$, we conclude that 
$\delta_\ell(\mathcal{H})\geq \frac{\alpha^k}{2^{k+\ell+1}}\binom{N-\ell}{k-\ell}$
with probability at least 
$1-\binom{N}{\ell}e^{-c_3N}$, which tends to $1$ as $N$ tends to infinity. Therefore, for all sufficiently large $N$, there exists an $\mathcal{S}_F$-free $k$-graph on $N$ vertices with a positive minimum $\ell$-degree, which means that $\pi_\ell(\mathcal{S}_F)>0.$ This completes the proof.

\subsection{Proof of \cref{j-vanishing}}

We first derive \cref{j-vanishing}\ref{cond2} from \cref{suspension}.
Let $F$ be a $k$-graph with $\pi_{\ell}(F)=0$.
Fix $S\in V(F)^{(\ell-1)}$ and let $L=L_F(S)$. Then $L$ is a $(k-\ell+1)$-graph. Let $\mathcal{S}^{\ell-1}_L$ denote the $(\ell-1)$-fold suspension of $L$ using the vertices of $S$ as apex vertices, 
that is, $E(\mathcal{S}^{\ell-1}_L)=\{S\cup e:e\in E(L)\}.$
Clearly, $\mathcal{S}^{\ell-1}_L\subseteq F$. 
Thus $\pi_\ell\bigl(\mathcal{S}^{\ell-1}_L\bigr)=0$ follows from that $\pi_\ell(F)=0$. 
Applying \cref{suspension} repeatedly gives that $\pi_1(L)=\pi(L)=0.$ 
By the theorem of Erd\H{o}s, an $r$-graph has ordinary Turán density zero
if and only if it is $r$-partite. Applying this with $r=k-\ell+1$, we conclude that  $L=L_F(S)$ is $(k-\ell+1)$-partite.

We now prove the contrapositive of \cref{j-vanishing}\ref{cond1}.
Suppose that $F$ is a $k$-graph with no
$(\ell+1)$-vanishing order. We next show that $\pi_\ell(F)>0$.
For this purpose, we construct a random $k$-graph $H$ on vertex set
$[n]$ as follows. 

Let $\Phi:[n]^{(\ell+1)}\to [k]^{(\ell+1)}$
be a uniformly random map, chosen independently on all $(\ell+1)$-sets.
For a $k$-set $e=\{t_1,\ldots,t_k\}\subseteq [n]$ where $t_1<\cdots<t_k$,
we declare $e$ to be an edge of $H$ if and only if
$\Phi\bigl(\{t_i:i\in \mathcal{L}\}\bigr)=\mathcal{L}$
for every $\mathcal{L}\in[k]^{(\ell+1)}.$
Then the natural order on $[n]$ clearly gives an $(\ell+1)$-vanishing order of $H$. 
Hence every subgraph of $H$ also has
an $(\ell+1)$-vanishing order, implying that $H$ is $F$-free.

It remains to show that $H$ has a positive minimum $\ell$-degree with positive probability for all sufficiently large $n$, and thus $\pi_\ell(F)>0$ follows.
Fix an $S\in [n]^{(\ell)}$, set $U=[n]\setminus S$, and write $m=n-\ell$.  Put
$ p=\binom{k}{\ell+1}^{-\binom{k}{\ell+1}}.$
For every $R\in U^{(k-\ell)}$, the $k$-set $S\cup R$ is an edge of $H$ with probability $p$, and hence
$\mathbb E (d_H(S))=p\binom{m}{k-\ell}.$
We now expose the random variables in blocks.  Fix an arbitrary ordering
$U=\{y_1,\ldots,y_m\}$.  For $1\le i\le m$, let
\[
\mathcal T_i=
\left\{T\in [n]^{(\ell+1)}:
 y_i\in T\setminus S \text{ and } y_i \text{ is the first vertex of }T\setminus S
 \text{ in the order } y_1,\ldots,y_m
\right\},
\]
and let $\mathcal X_i=(\Phi(T))_{T\in\mathcal T_i}$.  The blocks
$\mathcal X_1,\ldots,\mathcal X_m$ are independent and contain all random variables on which $d_H(S)$ may depend.
Resampling a single block $\mathcal X_i$ can change the status only of those completions $S\cup R$ with $y_i\in R$.  Therefore it changes $d_H(S)$ by at most
$D:=\binom{m-1}{k-\ell-1}.$
Applying \cref{lem:mcdiarmid} to the function $d_H(S)=f(\mathcal X_1,\ldots,\mathcal X_m)$ gives
\[
\mathbb P\left(d_H(S)<\frac p2\binom{m}{k-\ell}\right)
\le
\exp\left(-\frac{2(p\binom{m}{k-\ell}/2)^2}{mD^2}\right)
\le e^{-c n}
\]
for some constant $c=c(k,\ell)>0$.  Consequently,
\[
\mathbb P\left(\delta_\ell(H)<\frac p2\binom{n-\ell}{k-\ell}\right)
\le \binom n\ell e^{-cn}=o(1).
\]
Thus, for all sufficiently large $n$, there exists an $F$-free $k$-graph $H$ with
$\delta_\ell(H)\ge \frac p2\binom{n-\ell}{k-\ell},$
so $\pi_\ell(F)>0$.  This proves the contrapositive of \cref{j-vanishing}\ref{cond1}.


\section{Zero accumulation point: Proof of \cref{j-accumulation}}\label{upperbound}

We begin this section by proving \cref{lowdensity}. Should \cref{general} be valid, the desired statement \cref{j-accumulation} would follow as an immediate consequence. Regrettably, a proof of \cref{general} remains beyond our current scope. Fortunately, we may use \cref{suspension} to derive \cref{j-accumulation} via induction, with only \cref{2-vanishing} as our prerequisite.

\begin{theorem}\label{lowdensity}
Let $2\leq \ell\leq k-1$. Then for any $\varepsilon>0$ there exist $k$-graphs $F$ that satisfy $\pi_{\ell}(F)<\varepsilon$ while admitting no $\ell$-vanishing order.   
\end{theorem}
 
We need the following two lemmas to prove \cref{lowdensity}.

\begin{lemma}[Lo and Markstr\"{o}m, \cite{l-degree}]\label{supersaturation}   
Let $k > \ell \geq 1$ and $t>0$, and let $\mathcal{F}$ be a family of $k$-graphs with $\mathcal{F}(t)=\{F(t)~:~F\in\mathcal{F}\}$. 
Then $\pi_{\ell}(\mathcal{F})=\pi_{\ell}(\mathcal{F}(t))$.
\end{lemma}

\begin{lemma}[Lo and Markstr\"{o}m, \cite{l-degree}]\label{subset}   
For any integers $k > \ell \geq 1$ and any reals $\alpha, \varepsilon > 0$ with $\alpha + \varepsilon < 1$, there exists an integer $M(k, \ell, \varepsilon)$ such that the following holds. 
Let $n \geq m \geq M(k, \ell, \varepsilon)$ and let $H$ be a $k$-graph on $[n]$ with 
$\delta_\ell(H) \geq (\alpha + \varepsilon) \binom{n}{k-\ell}$.
Then the number of $m$-subsets $S$ of $[n]$ satisfying 
\[
\delta_\ell(H[S]) > \alpha \binom{m}{k-\ell}
\]
is at least $\frac{1}{2} \binom{n}{m}$.
\end{lemma}

\begin{proof}[Proof of \cref{lowdensity}]
To construct the required $k$-graphs, we consider the tensor product of all $k$-graphs on a fixed number of vertices with a large minimum $\ell$-degree. Precisely,  
for integers $m>k\geq 3$, let $\mathcal{F}_{m,\ell}^{(k)}:=\{F_1, F_2, \dots, F_s\}$ be the family of all $m$-vertex $k$-graphs with minimum $\ell$-degree at least $\binom{m-\ell-1}{k-\ell-1}+1$. 
Then we define ${F}_{m,\ell}^{(k)}$ to be the $k$-graph with vertex set $V(F_1)\times V(F_2)\times\dots\times V(F_s)$ such that $k$ vertices form an edge of ${F}_{m,\ell}^{(k)}$ if and only if for every $1\leq i\leq s$, the $i$-th coordinates of the $k$ vertices form an edge of $F_i$. 
The $k$-graph ${F}_{m,\ell}^{(k)}$ is sometimes referred to as the \textit{tensor product} of the elements of $\mathcal{F}^{(k)}_{m,\ell}$. 
A simple but useful fact is that $F_{m,\ell}^{(k)}$ is a subgraph of the blowup $F_i(m^{s-1})$ for any $1\leq i\leq s$.  

We first show that $F_{m,\ell}^{(k)}$ has no $\ell$-vanishing order for all integers $m>k>\ell\geq 2$.
For any two vertices $u,v\in V\left(F^{(k)}_{m,\ell}\right)$, we say that $u$ and $v$ are \textit{disjoint} if their $i$-th coordinates are distinct for every $1\leq i\leq s$. 
By the definition of $F^{(k)}_{m,\ell}$, any given $\ell$ vertices are contained in an edge if and only if they are pairwise disjoint. 
Let $\sigma$ be an arbitrary ordering of $V\left(F^{(k)}_{m,\ell}\right)$. 
We show that $\sigma$ cannot be an $\ell$-vanishing order of $F^{(k)}_{m,\ell}$. 

Let $v$ be the vertex with the smallest $\sigma(v)$ such that there exists a vertex $u$ with $\sigma(u)<\sigma(v)$ which is disjoint from $v$. 
Fix an edge $e\in E\left(F^{(k)}_{m,\ell}\right)$ containing $u$ and $v$, and assume that $e=\{x_1,x_2,\ldots,x_k\}$ with $\sigma(x_1)<\sigma(x_2)<\cdots<\sigma(x_k)$. 
By the choice of $v$, we have $x_1=u$ and $x_2=v$. 
For all $1\leq i\leq s$, since $\delta_\ell(F_i)\geq\binom{m-\ell-1}{k-\ell-1}+1$, there exists an edge $e_i\in F_i$ that contains all the $i$-th coordinates of $v,x_3,\ldots,x_{\ell+1}$ but does not contain the $i$-th coordinate of $u$. 
This implies that there is an edge $e'\in E\left(F^{(k)}_{m,\ell}\right)$ with 
$\{v,x_3,\ldots,x_{\ell+1}\}\subseteq e'$ such that every vertex in $e'$ is disjoint from $u$. 
By the choice of $v$, $\sigma(v)$ is the smallest among all vertices in $e'$. 
Setting $X=\{v,x_3,\ldots,x_{\ell+1}\}$, it follows that $\boldsymbol{I}_{X,\sigma}(e')\neq\boldsymbol{I}_{X,\sigma}(e)$, which implies by \cref{index} that $\sigma$ is not an $\ell$-vanishing order of $F^{(k)}_{m,\ell}$.  

To prove \cref{lowdensity}, it remains to show that, for every $\varepsilon>0$ and every
integer $k\ge3$, there exists $M>0$ such that
$\pi_\ell\left(F^{(k)}_{m,\ell}\right)<\varepsilon$
for every $m\ge M$.  It is enough to consider $0<\varepsilon<1$. Choose $M$ so large that $M\ge M(k,\ell,\varepsilon/8)$ from \cref{subset}, and such that for
all $m\ge M$,
$\frac{\varepsilon}{8}\binom{m}{k-\ell}
\geq \binom{m-\ell-1}{k-\ell-1}+1,$
and for all $n\ge M$,
$\binom{n-\ell}{k-\ell}\ge \frac12\binom{n}{k-\ell}.$
Let $n\ge m\ge M$, and let $H$ be an $n$-vertex $k$-graph with
$\delta_\ell(H)\ge \frac{\varepsilon}{2}\binom{n-\ell}{k-\ell}\ge \frac{\varepsilon}{4}\binom{n}{k-\ell}.$
Applying \cref{subset} with $\alpha=\varepsilon/8$ and error parameter $\varepsilon/8$, we get
an $m$-subset $S\subseteq V(H)$ such that
\[
\delta_\ell(H[S])>\frac{\varepsilon}{8}\binom{m}{k-\ell}
\geq \binom{m-\ell-1}{k-\ell-1}+1.
\]
Hence $H[S]\in \mathcal{F}^{(k)}_{m,\ell}$.  Therefore every sufficiently large $k$-graph with
minimum $\ell$-degree at least $(\varepsilon/2)\binom{n-\ell}{k-\ell}$ contains a member of
$\mathcal F^{(k)}_{m,\ell}$, and so
$\pi_\ell\left(\mathcal{F}^{(k)}_{m,\ell}\right)\le \frac{\varepsilon}{2}.$
By \cref{supersaturation},
$\pi_\ell\left(\mathcal{F}^{(k)}_{m,\ell}(m^{s-1})\right)
=\pi_{\ell}\left(\mathcal{F}^{(k)}_{m,\ell}\right)
\leq \frac{\varepsilon}{2}.$
Since $F_{m,\ell}^{(k)}$ is contained in $F(m^{s-1})$ for every
$F\in\mathcal{F}^{(k)}_{m,\ell}$, we conclude that
\[
\pi_\ell\left(F^{(k)}_{m,\ell}\right)
\leq \pi_\ell\left(\mathcal{F}^{(k)}_{m,\ell}(m^{s-1})\right)
\leq \frac{\varepsilon}{2}<\varepsilon,
\]
which completes the proof.
\end{proof}

We now give the proof of \cref{j-accumulation}.

\begin{proof}[Proof of \cref{j-accumulation}]
By \cref{2-vanishing}, we know that any $k$-graph without $2$-vanishing order has a positive $2$-degree Tur\'an density. 
Then it follows from \cref{lowdensity} that zero is an accumulation point of $\Pi_{2}^{k}$ for all $k\geq 3$. 
Therefore, to prove \cref{j-accumulation}, it suffices to show that for $\ell\geq3$, if $\Pi_{\ell-1}^{k-1}$ has zero as an accumulation point, then so does $\Pi_{\ell}^{k}$. 
For any $\varepsilon>0$, let $F$ be a $(k-1)$-graph with $0<\pi_{\ell-1}(F)<\varepsilon$. 
Then by \cref{suspension} we have $\pi_{\ell}(\mathcal{S}_F)>0$. 
On the other hand, $\pi_{\ell}(\mathcal{S}_F)\le \pi_{\ell-1}(F)<\varepsilon$, and hence zero is an accumulation point of $\Pi_{\ell}^{k}$.   
\end{proof}

\section{Concluding remarks}\label{remark}
In this paper, we proved that any $k$-graph $F$ with $\pi_{2}(F)=0$ admits a $2$-vanishing order, thus providing a necessary condition for a $k$-graph to have vanishing $2$-degree Tur\'an density. 
A natural follow-up question is how close this necessary condition is to being sufficient.
From the proof of \cref{2-vanishing}, one can see that $(2,1^{k-2})$-type $k$-graphs play an important role, and we in fact obtain a $2$-vanishing cluster order. It would be interesting to test on this family of hypergraphs.

\begin{problem}
    Characterize $(2,1^{k-2})$-type $k$-graphs with vanishing $2$-degree Tur\'an density.
\end{problem}

\cref{general} has a connection to uniform Tur\'an densities defined as follows. 
For reals $d\in[0,1], \eta>0$ and integers $0\leq \ell< k$, an $n$-vertex $k$-graph $H$ is \textit{uniformly $(d,\eta,\ell)$-dense} if 
$|\mathcal{K}_k(G)\cap E(H)|\geq d|\mathcal{K}_k(G)|-\eta n^k$
holds for all $\ell$-graphs $G$ on vertex set $V(H)$, 
where $\mathcal{K}_k(G)$ denotes the collection of all $k$-subsets of $V(G)$ each of which spans a clique $K_k^{(\ell)}$ of $G$.
Given a $k$-graph $F$, the \textit{$\ell$-uniform Tur\'an density} of $F$ is defined as
\begin{align*}\label{j-turan-dense}
\pi_{\rm u}^{(\ell)}(F) = \sup \{ d\in [0,1] & : \text{for\ every\ } \eta>0,\ \text{there\ exist\ infinitely\ many}  \\
&\quad   F \text{-free\ uniformly} \ (d,\eta,\ell) \text{-dense}\ k \text{-graphs\ }\}.
\end{align*}
It is known that for any $k$-graph $F$
$$0=\pi_{\rm u}^{(k-1)}(F)\leq\pi_{\rm u}^{(k-2)}(F)\leq\cdots\leq\pi_{\rm u}^{(1)}(F)\leq\pi_{\rm u}^{(0)}(F)=\pi(F).$$
Particularly, $1$-uniform Tur\'an density is just the uniform Tur\'an density $\pi_{\points}$.
Such problems were first suggested by Erd\H{o}s and S\'{o}s \cite{Erdos-sos} for $3$-graphs and have been extensively studied in the past decade. See, for instance, ~\cite{orderfive,cycledan,1/27dan,K43-dan,ander,zhou2,reihersurvey,reiher2018hypergraphs,K43-rodl}. 

An interesting conjecture of Reiher, R\"{o}dl and Schacht, \cite{reiher2018hypergraphs} relates vanishing order to zero uniform Tur\'an density.

\begin{conjecture}[\cite{reiher2018hypergraphs}]\label{generaluniform}
 For a $k$-graph $F$ and integers $1\leq \ell\leq k-3$, $\pi_{\rm u}^{(\ell)}(F)=0$ if and only if $F$ has an $(\ell+1)$-vanishing order. 
\end{conjecture}

If \cref{generaluniform} holds, then \cref{2-vanishing} would imply that for all $k\ge 3$, the condition $\pi_2(F)=0$ forces $\pi_{\rm u}^{(1)}(F)=\pi_{\points}(F)=0$.

\medskip
\noindent{\bf Acknowledgements.~}The authors would like to thank David Conlon, Guanghui Wang and Shuaichao Wang for helpful discussions.

\bibliographystyle{abbrv}
\bibliography{references.bib}

@article{orderfive,
  title={Tur\'{a}n density of cliques of order five in 3-uniform hypergraphs with quasirandom links},
  author={Berger, S\"{o}ren and Piga, Sim\'{o}n and Reiher, Christian and R{\"o}dl, Vojt{\v{e}}ch and Schacht, Mathias},
  journal={Transactions of the American Mathematical Society},
  volume={378},
  number={12},
  pages={8283-8318},
  year={2025},
  publisher={}
}

@article{K43-dan,
  title={A problem of {Erd\H{o}s and S{\'o}s} on 3-graphs},
  author={Glebov, Roman and Kr{\'a}l', Daniel and Volec, Jan},
  journal={Israel Journal of Mathematics},
  volume={211},
  number={1},
  pages={349-366},
  year={2016},
  publisher={}
}

@article{K43-rodl,
  title={On a {Tur{\'a}n} problem in weakly quasirandom 3-uniform hypergraphs},
  author={Reiher, Christian and R{\"o}dl, Vojt{\v{e}}ch and Schacht, Mathias},
  journal={Journal of the European Mathematical Society},
  volume={20},
  number={5},
  pages={1139-1159},
  year={2018},
  publisher={}
}

@article{reihersurvey,
  title={Extremal problems in uniformly dense hypergraphs},
  author={Reiher, Christian},
  journal={European Journal of Combinatorics},
  volume={88},
  number={},
  pages={103117},
  year={2020},
  publisher={}
}

@article{l-degree,
  title={$\ell$-degree {Tur\'{a}n} density},
  author={Lo, Allan and Markstr\"{o}m, Klas},
  journal={SIAM Journal on Discrete Mathematics},
  volume={28},
  number={3},
  pages={1214-1225},
  year={2014},
  publisher={}
}

@article{reiher2018hypergraphs,
  title={Hypergraphs with vanishing {Tur{\'a}n} density in uniformly dense hypergraphs},
  author={Reiher, Christian and R{\"o}dl, Vojt{\v{e}}ch and Schacht, Mathias},
  journal={Journal of the London Mathematical Society},
  volume={97},
  number={1},
  pages={77--97},
  year={2018},
  publisher={Wiley Online Library}
}

@article{1/27dan,
  title={Hypergraphs with minimum positive uniform {Tur{\'a}n} density},
  author={Garbe, Frederik and Kr{\'a}l', Daniel and Lamaison, Ander},
  journal={Israel Journal of Mathematics},
  volume={259},
  pages={701-726},
  year={2024},
  publisher={Springer}
}

@article{zhou2,
  title={The minimum positive uniform {Tur{\'a}n} density in uniformly dense $k$-uniform hypergraphs},
  author={Lin, Hao and Wang, Guanghui and Zhou, Wenling},
  journal={SIAM Journal on Discrete Mathematics},
  volume={39},
  number={4},
  pages={2350-2378},
  year={2025},
  publisher={}
}

@article{l2norm,
  title={Hypergraph {Tur\'{a}n} problems in $\ell_2$-norm},
  author={Balogh, J\'{o}zsef and Clemen, Felix Christian and Lidick\'{y}, Bernard},
  journal={London Mathematical Society Lecture Note Series},
  volume={481},
  number={},
  pages={21--63},
  year={2022},
  publisher={}
}

@article{James,
  title={Hypergraphs of arbitrary uniformity with vanishing codegree {Tur\'{a}n} density},
  author={Sarkies, James},
  journal={arXiv:2503.23591},
  year={2025}
}

@article{Luo-cycle,
  title={A jump in the codegree {Tur\'{a}n} densities of long tight cycles},
  author={Balogh, J\'{o}zsef and Luo, Haoran and Sankar, Maya},
  journal={arXiv:2602.18398},
  year={2026}
}

@article{Rong-cycle,
  title={The codegree {Tur\'{a}n} density of tight cycles},
  author={Ma, Jie and Rong, Mingyuan},
  journal={arXiv:2512.23011},
  year={2025}
}

@article{cycledan,
  title={Uniform {Tur{\'a}n} density of cycles},
  author={Buci{\'c}, Matija and Cooper, Jacob W. and Kr{\'a}l', Daniel and Mohr, Samuel and Correia, David Munh{\'a}},
  journal={Transactions of the American Mathematical Society},
  volume={376},
  number={7},
  pages={4765-4809},
  year={2023},
  publisher={}
}

@article{wilson1975existence,
  title={{An} existence theory for pairwise balanced designs, {III}: Proof of the existence conjectures},
  author={Wilson, Richard M},
  journal={Journal of Combinatorial Theory, Series A},
  volume={18},
  number={1},
  pages={71--79},
  year={1975},
  publisher={Elsevier}
}

@article{ding20243,
  title={On $3$-graphs with vanishing codegree {Tur\'{a}n} density},
  author={Ding, Laihao and Lamaison, Ander and Liu, Hong and Wang, Shuaichao and Yang, Haotian},
  journal={Journal of the London Mathematical Society (2)},
  volume={},
  year={2025, {DOI}: 10.1112/jlms.70281}
}

@article{ander,
  title={Palettes determine uniform {Tur\'{a}n} density},
  author={Lamaison, Ander},
  journal={arXiv:2408.09643},
  volume={},
  year={2024}
}

@article{keevash2007codegree,
  title={Codegree problems for projective geometries},
  author={Keevash, Peter and Zhao, Yi},
  journal={Journal of Combinatorial Theory, Series B},
  volume={97},
  number={6},
  pages={919--928},
  year={2007},
  publisher={Elsevier}
}

@article{lo2018codegree,
      title={Codegree {Tur\'{a}n} density of complete $r$-uniform hypergraphs}, 
      author={Allan Lo and Yi Zhao},
      journal={SIAM Journal on Discrete Mathematics},
  volume={32},
  number={2},
  pages={1154-1158},
  year={2018},
  publisher={}
}

@article{piga2024codegree,
  title={The codegree {Tur\'{a}n} density of $3$-uniform tight cycles},
  author={Piga, Sim{\'o}n and Sanhueza-Matamala, Nicol{\'a}s and Schacht, Mathias},
  journal={Journal of Combinatorial
Theory, Series B},
  volume={176},
  number={},
  pages={1-6},
  year={2026}
}

@article{ma2024codegree,
  title={On codegree {Tur\'{a}n} density of the 3-uniform tight cycle {$C_{11}$}},
  author={Ma, Jie},
  journal={arXiv:2409.02765},
  year={2024}
}

@article{piga2023hypergraphs,
  title={Hypergraphs with arbitrarily small codegree {Tur{\'a}n} density},
  author={Piga, Sim{\'o}n and Sch{\"u}lke, Bjarne},
  journal={Bulletin of the London Mathematical Society},
  volume={58},
  number={4},
  pages={e70348},
  year={2026}
}

@article{Accu-Turan-Bjarne,
  title={Hypergraphs Accumulate},
  author={Conlon, David and Sch{\"u}lke, Bjarne},
  journal={International Mathematics Research Notice},
  volume={2025},
  number={2},
  pages={rnae289},
  year={2025}
}

@article{erdos-sos,
  title={On {Ramsey–Tur{\'a}n} type theorems for hypergraphs},
  author={Erd\H{o}s, Paul and S{\'o}s, Vera T.},
  journal={Combinatorica},
  volume={2},
  number={3},
  pages={289–295},
  year={1982},
  publisher={}
}

@article{K43-codegree,
  title={The codegree threshold of {$K_4^-$}},
  author={Falgas-Ravry, Victor and Pikhurko, Oleg and Vaughan, Emil and Volec, Jan},
  journal={Journal of the London Mathematical Society (2)},
  volume={107},
  pages={1660-1691},
  year={2023},
  publisher={}
}

@article{erdos1964extremal,
  title={On extremal problems of graphs and generalized graphs},
  author={Erd{\H o}s, Paul},
  journal={Israel Journal of Mathematics},
  volume={2},
  number={3},
  pages={183--190},
  year={1964},
  publisher={Springer}
}

@article{zhang2021codegree,
  title={On the Codegree Density of {$PG_m(q)$}},
  author={Zhang, Tao and Ge, Gennian},
  journal={SIAM Journal on Discrete Mathematics},
  volume={35},
  number={3},
  pages={1548--1556},
  year={2021},
  publisher={SIAM}
}

@article{mubayi2005co,
  title={The co-degree density of the {Fano} plane},
  author={Mubayi, Dhruv},
  journal={Journal of Combinatorial Theory, Series B},
  volume={95},
  number={2},
  pages={333--337},
  year={2005},
  publisher={Elsevier}
}

@article{falgas2015codegree,
  title={The codegree threshold for 3-graphs with independent neighborhoods},
  author={Falgas--Ravry, Victor and Marchant, Edward and Pikhurko, Oleg and Vaughan, Emil R},
  journal={SIAM Journal on Discrete Mathematics},
  volume={29},
  number={3},
  pages={1504--1539},
  year={2015},
  publisher={SIAM}
}

@article{piga2022codegree,
  title={The codegree {Tur{\'a}n} density of tight cycles minus one edge},
  author={Piga, Sim{\'o}n and Sales, Marcelo and Sch{\"u}lke, Bjarne},
  journal={Combinatorics, Probability and Computing},
  volume={32},
  number={},
  pages={881-884},
  year={2023},
  publisher={}
}

@article{K43nagle,
  title={A note on codegree problems for hypergraphs},
  author={Czygrinow, Andrzej and Nagle, Brendan},
  journal={Bulletin of the Institute of Combinatorics and its Applications},
  volume={32},
  number={},
  pages={63-69},
  year={2001},
  publisher={}
}

@article{K43-nagle,
  title={{Tur{\'a}n}-related problems for hypergraphs},
  author={Nagle, Brendan},
  journal={Congressus Numerantium},
  volume={},
  number={},
  pages={119–128},
  year={1999},
  publisher={}
}

@article{McDiarmid2002,
  author  = {McDiarmid, Colin},
  title   = {Concentration for Independent Permutations},
  journal = {Combinatorics, Probability and Computing},
  volume  = {11},
  number  = {2},
  pages   = {163--178},
  year    = {2002},
  doi     = {10.1017/S0963548301005089}
}

@incollection{McDiarmid1989,
  author    = {McDiarmid, Colin},
  title     = {On the Method of Bounded Differences},
  booktitle = {Surveys in Combinatorics 1989},
  series    = {London Mathematical Society Lecture Note Series},
  volume    = {141},
  pages     = {148--188},
  publisher = {Cambridge University Press},
  year      = {1989}
}
\end{document}